\documentclass[10pt]{article}

\usepackage{amsmath}
\usepackage{amssymb}
\usepackage{latexsym}
\usepackage{theorem}
\usepackage[all]{xy}
\usepackage{color}
\usepackage{stmaryrd}

\newenvironment{proof}{\par \noindent{\bf Proof: }}{\hspace{\stretch{1}} $\Box$ \par \mbox{}}

\newtheorem{theorem}{Theorem}[section]
\newtheorem{proposition}[theorem]{Proposition}
\newtheorem{lemma}[theorem]{Lemma}
\newtheorem{corollary}[theorem]{Corollary}
{\theorembodyfont{\rmfamily}
\newtheorem{definition}[theorem]{Definition}
\newtheorem{example}[theorem]{Example}
}
\newenvironment{theorem*}{\par \medskip \noindent{\bf Theorem }}{\par \mbox{}}
\newenvironment{lemma*}{\par \medskip \noindent{\bf Theorem }}{\par \mbox{}}

\newcommand{\Hom}{\mathop{\rm Hom}}

\newcommand{\Ob}{\mathop{\rm Ob}}

\newcommand{\gtp}{\mathop{\hat{\otimes}}}
\newcommand{\F}{\mathbb{F}}
\newcommand{\R}{\mathbb{R}}
\newcommand{\C}{\mathbb{C}}
\newcommand{\K}{\mathbb{K}}

\newcommand{\N}{\mathbb{N}}

\newcommand{\id}{\mathop{\rm id}}
\let\svthefootnote\thefootnote
\newcommand\freefootnote[1]{%
  \let\thefootnote\relax%
  \footnotetext{#1}%
  \let\thefootnote\svthefootnote%
}

\begin{document}

\title{$E$-theory for $C^\ast$-categories}

\author{Sarah L. Browne and Paul D. Mitchener}

%

\maketitle

\begin{abstract}
$E$-theory was originally defined concretely by Connes and Higson~\cite{CH} and further work followed this construction. We generalise the definition to $C^\ast$-categories. $C^\ast$-categories were formulated to give a theory of operator algebras in a categorical picture and play important role in the study of mathematical physics. In this context, they are analogous to $C^\ast$-algebras and so have invariants defined coming from $C^\ast$-algebra theory but they do not yet have a definition of $E$-theory. Here we define $E$-theory for both complex and real graded $C^\ast$-categories and prove it has similar properties to $E$-theory for $C^\ast$-algebras. 
\end{abstract}

\freefootnote{
\noindent \hspace{-6.5mm}
\newline
SB was partially supported by NSF grant \#DMS--1564401
\newline
Contact emails: slbrowne@ku.edu, p.mitchener@sheffield.ac.uk}

\section{Introduction}
Throughout this article, we use $\F$ to denote either the field of real numbers $\R$ or the field of complex numbers $\C$.
$C^\ast$-categories are analogous to $C^\ast$-algebras but give a much more general framework which do not require a choice of Hilbert space. In particular, we can consider all Hilbert spaces and the collection of all bounded linear operators on them simultaneously, and we can view a $C^\ast$-algebra as a one object $C^\ast$-category. 

$E$-theory is an invariant for $C^\ast$-algebras. $K$-theory for $C^\ast$-algebras was generalised to $C^\ast$-categories in \cite{Jo2,Mitch2.5} and we take a similar approach to generalising $E$-theory here. We build on the work of Guentner-Higson~\cite{GH04} and Browne~\cite{Bro17} on constructing $E$-theory for graded complex and real $C^\ast$-categories. 

The plan in future papers is to build spectra related to $E$-theory, and eventually equivariant $E$-theory. With these defined, we will be ready to develop applications in the field of analytic assembly. Ultimately this machinery should give a smooth proof that the Baum-Connes conjecture implies the stable Gromov-Lawson-Rosenberg conjecture, as in \cite{Stolz1}, and similar results that require looking at analytic assembly from a homotopy-theoretic point of view.

This paper details the generalisation of a structure of asymptotic morphisms to asymptotic functors for $C^\ast$-categories, and then the notion of homotopy for these. Then we construct an abelian group structure and define $E$-theory for complex and real graded $C^\ast$-categories. We then construct a product and prove that $E$-theory is a bivariant functor going from the category where objects are $C^\ast$-categories and morphisms are $\ast$-functors to the category where objects are abelian groups and morphisms are group homomorphisms.  

We check properties of standard $E$-theory also hold in this context. Namely we check that the functor $E$ is half exact, homotopy invariant and has Bott periodicity. We give details and show that we have long exact sequences induced from short exact sequences.

\section{$C^\ast$-categories and asymptotic functors}
In this section we give a review on the notions of $C^\ast$-category, representations and ideals in this setting before giving the definition of an asymptotic functor and the notion of homotopy of these which will form part of the definition of $E$-theory for $C^\ast$-categories.

Recall (see \cite{GLR, Mitch2}) that a \emph{unital $C^\ast$-category} over the field $\F$ is a category $\mathcal A$ where:
\begin{itemize}
\item Each morphism set $\Hom (a,b)_{\mathcal A}$ is a Banach space over the field $\F$. Composition of morphisms
\[ \Hom (b,c)_{\mathcal A} \times \Hom (a,b)_{\mathcal A} \rightarrow \Hom (a,c)_{\mathcal A} \]
is bilinear and satisfies the inequality $\| xy \| \leq \| x\| \cdot \| y \|$ for all $x\in \Hom (b,c)_{\mathcal A}$ and $y\in \Hom (a,b)_{\mathcal A}$.
\item We have an {\em involution}, that is to say conjugate-linear maps $\Hom (a,b)_{\mathcal A}\rightarrow \Hom (b,a)_{\mathcal A}$, written $x\mapsto x^\ast$ such that $(x^\ast )^\ast =x$ for all $x\in \Hom (a,b)_{\mathcal A}$, and $(xy)^\ast =y^\ast x^\ast$ for all $x\in \Hom (b,c)_{\mathcal A}$ and $y\in \Hom (a,b)_{\mathcal A}$.
\item The {\em $C^\ast$-identity} $\| x^\ast x \| = \| x\|^2$ holds for all $x\in \Hom (a,b)_{\mathcal A}$. Further, the element $x^\ast x$ is a positive element of the unital $C^\ast$-algebra $\Hom (a,a)_{\mathcal A}$.
\end{itemize}

Similarly we can define a non-unital $C^\ast$-category $\mathcal A$ by dropping the insistence of unit elements $1_a\in \Hom (a,a)_{\mathcal A}$. Hence a non-unital $C^\ast$-category is not actually a category, but there are objects and morphisms with an associative composition rule that satisfy the above axioms.
The key example of a $C^\ast$-category is $\mathcal L$, the category of all Hilbert spaces and bounded linear operators. The norm on each morphism set is the operator norm, and the involution is defined by taking adjoints.

The following comes from \cite{Mitch2}.

\begin{definition}
Let $\mathcal A$ be a $C^\ast$-category, and let $\mathcal J$ be a collection of objects and morphisms which are closed under composition (effectively, $\mathcal J$ is a``non-unital subcategory"). We call $\mathcal J$ a {\em $C^\ast$-ideal} if:

\begin{itemize}

\item  $\Ob ({\mathcal J}) = \Ob ({\mathcal A})$;

\item Each space $\Hom (a,b)_{\mathcal J}$ is a vector subspace of $\Hom (a,b)_{\mathcal A}$;

\item If $x\in \Hom (a,b)_{\mathcal J}$, then $x^\ast \in \Hom (b,a)_{\mathcal J}$;

\item For all $x\in \Hom (a,b)_{\mathcal J}$, $u\in Hom(a',a)_{\mathcal A}$, and $v\in Hom (b,b')_{\mathcal B}$ the composites $vx$ and $xu$ are morphisms in the category $\mathcal J$.

\end{itemize}

\end{definition}

It is shown in \cite{Mitch2} that the morphism sets of a $C^\ast$-ideal are automatically closed under the norm, so any $C^\ast$-ideal is a $C^\ast$-category. Further, we can form the quotient ${\mathcal A}/{\mathcal J}$; the objects are the same as those of the categories $\mathcal A$ and $\mathcal J$, and the morphism set $\Hom (a,b)$ in the quotient ${\mathcal A}/{\mathcal J}$ is the quotient Banach space ${\Hom (a,b)_\mathcal A}/\Hom (a,b)_{\mathcal J}$.

If $\mathcal A$ and $\mathcal B$ are $C^\ast$-categories, a \emph{$\ast$-functor} $\alpha \colon {\mathcal A}\rightarrow {\mathcal B}$ consists of a map $\alpha \colon \Ob ({\mathcal A})\rightarrow \Ob ({\mathcal B})$ and linear maps $\alpha \colon \Hom (a,b)_{\mathcal A}\rightarrow \Hom (\alpha (a),\alpha (b))_{\mathcal A}$ such that:
\begin{itemize}
\item $\alpha (xy) =\alpha (x)\alpha (y)$ for all $x\in \Hom (b,c)_{\mathcal A}$ and $y\in \Hom (a,b)_{\mathcal A}$.
\item $\alpha (x^\ast ) = \alpha (x)^\ast$ for all $x\in \Hom (a,b)_{\mathcal A}$.
\end{itemize}

For example, if $\mathcal A$ is a $C^\ast$-category and $\mathcal J$ is a $C^\ast$-ideal, we have a {\em quotient} $\ast$-functor $\pi \colon {\mathcal A}\rightarrow {\mathcal A}/{\mathcal J}$ defined by taking $\pi$ to be the identity map on the set of objects, and the quotient map $\pi \colon \Hom (a,b)_{\mathcal A}\rightarrow \Hom (a,b)_{\mathcal A}/\Hom (a,b)_{\mathcal J}$ on each morphism set.

It is shown in \cite{Mitch2} that any $\ast$-functor is norm-decreasing, that is to say $\| \alpha (x) \|\leq \| x\|$ for all $x\in \Hom (a,b)_{\mathcal A}$. In particular, each map $\alpha \colon \Hom (a,b)_{\mathcal A}\rightarrow \Hom (\alpha (a),\alpha (b))_{\mathcal B}$ is continuous. If the $\ast$-functor $\alpha$ is {\em faithful}, that is to say injective on each morphism set, then it is an isometry. It is shown in \cite{GLR,Mitch2} that for any $C^\ast$-category $\mathcal A$ there is a faithful $\ast$-functor $\rho \colon {\mathcal A}\rightarrow {\mathcal L}$; we call such a $\ast$-functor a {\em representation}.

A $\ast$-functor is a logical generalisation of a $\ast$-homomorphism between $C^\ast$-categories. We can also generalise asymptotic morphisms to the notion of an asymptotic functor for $C^\ast$-categories as below.

\begin{definition}
Let $\mathcal A$ and $\mathcal B$ be $C^\ast$-categories. An {\em asymptotic functor} $\varphi = \varphi _t\colon {\mathcal A}\dashrightarrow {\mathcal B}$ consists of:
\begin{itemize}
\item $\varphi \colon \Ob ({\mathcal A})\rightarrow \Ob ({\mathcal B})$,
\item For each $t\in [1,\infty )$ a map $\varphi_t \colon \Hom (a,b)_{\mathcal A} \rightarrow \Hom (\varphi (a),\varphi (b))_{\mathcal B}$,
\end{itemize}
with the following properties:
\begin{itemize}
\item For each $x\in \Hom (a,b)_{\mathcal A}$ the map $[1,\infty )\rightarrow \Hom (\varphi (a),\varphi (b))_{\mathcal B}$ defined by writing $t\mapsto \varphi_t (x)$ is continuous and bounded.
\item For all $x,y\in \Hom (a,b)_{\mathcal A}$ and $\lambda , \mu \in \F$ we have
\[ \lim_{t\rightarrow \infty} \|  \varphi_t (\lambda x + \mu y) - \lambda \varphi_t (x) - \mu \varphi_t (y) \| =0 .\]
\item For all $x\in \Hom (b,c)_{\mathcal A}$ and $y\in \Hom (a,b)_{\mathcal A}$ we have
\[ \lim_{t\rightarrow \infty} \| \varphi_t (xy) - \varphi_t(x)\varphi_t (y) \| =0 .\]
\item For all $x\in \Hom (a,b)_{\mathcal A}$ we have
\[ \lim_{t\rightarrow \infty} \| \varphi_t (x^\ast ) -\varphi_t (x)^\ast \| = 0. \]
\end{itemize}
\end{definition}

A $\ast$-functor $\alpha \colon {\mathcal A}\rightarrow {\mathcal B}$ can also be considered as an asymptotic functor by writing $\alpha_t (x) = \alpha (x)$ for all $t\in [1,\infty )$ and $x\in \Hom (a,b)_{\mathcal A}$.

As noted in \cite{Mitch2} we can form the spatial tensor product of $C^\ast$-categories. We first form the algebraic tensor product.  If $a\in \Ob ({\mathcal A})$ and $b\in \Ob ({\mathcal B})$ let us write the pair $(a,b)\in \Ob ({\mathcal A})\times \Ob ({\mathcal B})$ as $a\otimes b$. Then the algebraic tensor product is an algebroid (see \cite{Mi}) ${\mathcal A}\odot {\mathcal B}$, with objects $\Ob ({\mathcal A})\times \Ob ({\mathcal B})$. The morphism sets are algebraic tensor products of vector spaces
\[ \Hom (a\otimes b,a'\otimes b')_{{\mathcal A}\odot {\mathcal B}} = \Hom (a,a')_{\mathcal A}\odot \Hom (b,b')_{\mathcal B} .\]

Composition of morphisms and the involution are defined by the formulae
\begin{itemize}
\item $(u\otimes v)(x\otimes y) = ux\otimes vy$ where $x\in \Hom (a,a')_{\mathcal A}, u\in \Hom (a',a'')_{\mathcal A}$ , and $y\in \Hom (b,b')_{\mathcal B}$, $v\in \Hom (b',b'')_{\mathcal B}$,
\item $(x\otimes y)^\ast = x^\ast \otimes y^\ast$ where $x\in \Hom (a,a')_{\mathcal A}$ and $y\in \Hom (b,b')_{\mathcal B}$,
\end{itemize}
and extending by linearity.

Let $H$ and $H'$ be Hilbert spaces. Consider the algebraic tensor product $H\odot H'$ with bilinear form defined by
\[ \langle v\otimes v',w\otimes w' \rangle = \langle v,w\rangle \cdot \langle v',w'\rangle \qquad v,w\in H, v',w'\in H' . \]

We then take the quotient by elements $x\in H\odot H'$ such that $\langle x,x \rangle =0$ and complete with respect to the resulting norm to obtain the tensor product $H\otimes H'$.

Pick faithful representations $\rho_{\mathcal A} \colon {\mathcal A}\rightarrow {\mathcal L}$ and $\rho_{\mathcal B} \colon {\mathcal B}\rightarrow {\mathcal L}$. As noted above, such always exist, and are isometries on each morphism set. We can define a faithful $\ast$-functor $\rho \colon {\mathcal A}\odot {\mathcal B}\rightarrow {\mathcal L}$ by mapping the object $a\otimes b$ to the Hilbert space $\rho (a\otimes b) = \rho_{\mathcal A}(a)\otimes \rho_{\mathcal B}(b)$, and the morphism $x\otimes y$ where $x\in \Hom (a,a')_{\mathcal A}$ and $y\in \Hom (b,b')_{\mathcal B}$ to the bounded linear map $\rho (x\otimes y) \colon \rho (a\otimes b)\rightarrow \rho (a'\otimes b')$ defined by the formula $\rho (x\otimes y)(v\otimes w) = \rho_{\mathcal A}(x)(v) \otimes \rho_{\mathcal B}(y)(w)$ where $v\otimes w\in \rho (a\otimes b)$. We extend where necessary by linearity.

It follows that we can define a norm on the morphism sets of the algebraic tensor product by writing $\| z \| = \| \rho (z) \|$ for each morphism $z$ in ${\mathcal A}\odot {\mathcal B}$. We define the {\em spatial tensor product} ${\mathcal A}\otimes {\mathcal B}$ by completion of each morphism set in the algebraic tensor product under this norm. As shown in \cite{Mitch2}, the spatial tensor product ${\mathcal A}\otimes {\mathcal B}$ is a $C^\ast$-category, and the norm does not depend on our choices of faithful representations.

\begin{example}
Let $X$ be a compact Hausdorff space. Then the $C^\ast$-algebra $C_0(X)$ can be viewed as a $C^\ast$-category with one object. Let $\mathcal A$ be another $C^\ast$-category. Then the tensor product $C_0(X)\otimes {\mathcal A}$ is a $C^\ast$-category with the same objects as $\mathcal A$ and morphism sets
\[ \Hom (a,b)_{C_0(X)\otimes {\mathcal A}} = \{ f\colon X\rightarrow \Hom (a,b)_{\mathcal A} \ |\ f \textrm{ continuous}, \lim_{x \to \infty} \| f(x)\| \rightarrow 0\} \]
\end{example}

\begin{proposition}
Let $\varphi \colon {\mathcal A}\dashrightarrow {\mathcal B}$ be an asymptotic functor. Let $\mathcal C$ be a $C^\ast$-category. Then there is an asymptotic functor
\[ \varphi \otimes \id \colon {\mathcal A}\otimes {\mathcal C} \dashrightarrow {\mathcal B}\otimes {\mathcal C} \]
such that $(\varphi \otimes \id )(a\otimes c) = \varphi (a)\otimes c$ for all $a\in \Ob ({\mathcal A})$ and $c\in \Ob ({\mathcal C})$, and
\[ (\varphi \otimes \id )_t (x\otimes y) = \varphi_t (x)\otimes y \]
if $x\in \Hom (a,a')_{\mathcal A}$, $y\in \Hom (c,c')_{\mathcal C}$ and $t\in [1,\infty )$.
\end{proposition}
\begin{proof}
We need to check the asymptotic properties of the maps $(\varphi \otimes \id )_t \colon \Hom (a\otimes c,b\otimes c)_{{\mathcal A}\otimes {\mathcal C}} \rightarrow \Hom (\varphi (a), \varphi (b))_{{\mathcal B}\otimes {\mathcal C}}$.

Firstly, let $x\otimes y \in \Hom (a,b)_{\mathcal A}\otimes \Hom (c,c')_ {\mathcal C}$. We know the map $[0,1)\rightarrow \Hom (\varphi (a),\varphi (b))_{\mathcal A}$ given by $t\mapsto \varphi_t (x)$ is continuous and bounded. The map $t\mapsto \varphi_t(x)\otimes y$ is certainly continuous. Suppose $\| \varphi_t (x)\| \leq M$ for all $t\in [1,\infty )$. Then $\| \varphi_t (x) \otimes y \| \leq M\| y\|$ by definition of the spatial tensor product, and the map $t\mapsto \varphi_t(x)\otimes y$ is certainly bounded.

Let $x,y\in \Hom (a,b)_{\mathcal A}$, $\lambda , \mu \in \F$. Let $z\in \Hom (c,c')_{\mathcal C}$. Then
\[ \| \varphi_t (\lambda x+\mu y )\otimes z - \lambda \varphi_t (x) \otimes z - \mu \varphi_t (y) \otimes z \| \leq \| z\|  \| \varphi_t (\lambda x+\mu y )- \lambda \varphi_t (x)  - \mu \varphi_t (y)  \| \] 
which converges to $0$ as $t\rightarrow \infty$.

The other required limit properties follow similarly.
\end{proof}

We can similarly define asymptotic functors of the form $\id \otimes \varphi \colon {\mathcal C}\otimes {\mathcal A}\rightarrow {\mathcal C}\otimes {\mathcal B}$.

\begin{definition}
Let $\alpha \colon {\mathcal A}\rightarrow {\mathcal B}$ be a $\ast$-functor. Then the {\em mapping cylinder} ${\mathcal C}_\alpha$ is the $C^\ast$-category with objects
\[ \Ob ({\mathcal C}_\alpha ) = \{ (a,b)\in \Ob ({\mathcal A})\times \Ob ({\mathcal B}) \ |\ \alpha (a) = b \} \]
and morphism sets
\[ \Hom ((a,b),(a',b')) = \left\{ (x,f) \; \middle|\ \; \Large\begin{subarray}{c} x \; \in \; \Hom (a,a')_{\mathcal A}, \ f(0)= \alpha (x), \ \vspace{2mm}  \\ f\colon  [0,1] \; \longrightarrow  \; \Hom (b,b')_{\mathcal B} \textrm{ continuous}    
   \end{subarray}\right\} . \]
\end{definition}

The mapping cylinder fits into a pullback diagram 
\[\xymatrixcolsep{3pc}\xymatrixrowsep{3pc} \xymatrix{ {\mathcal C}_\alpha \ar[d] \ar[r] & {\mathcal A} \ar[d]^{\alpha} \\
C[0,1]\otimes {\mathcal B} \ar[r]_-{E_0} & {\mathcal B} \\}\]
in the category of $C^\ast$-categories and $\ast$-functors. 

In the special case where the categories $\mathcal A$ and $\mathcal B$ have the same set of objects, and the $\ast$-functor $\alpha$ is the identity on the set of objects, the mapping cylinder ${\mathcal C}_\alpha$ has a simpler description. The set of objects of ${\mathcal C}_\alpha$ can be the same as the set of the objects of the categories $\mathcal A$ and $\mathcal B$. The morphism set $\Hom (a,b)$ in the mapping cylinder ${\mathcal C}_\alpha$ is the same as the mapping cylinder of the map $\alpha \colon \Hom (a,b)_{\mathcal A}\rightarrow \Hom (a,b)_{\mathcal B}$, that is to say the set
\[ \{ (x,f) \ |\ x\in \Hom (a,b)_{\mathcal A}, \ f\colon [0,1]\rightarrow \Hom (a,b)_{\mathcal B} \textrm{ continuous},\ f(0)= \alpha (x) \} . \]

Using tensor products, we can form a notion of homotopy of asymptotic functors. Specifically, if $\mathcal A$ is a $C^\ast$-category, define the $C^\ast$-category $I{\mathcal A} = {\mathcal A}\otimes C[0,1]$. There are $\ast$-homomorphisms $e_0 ,e_1 \colon C[0,1]\rightarrow \F$ defined by writing $e_0 (f) =f(0)$ and $e_1 (f) =f(1)$ respectively. Hence, by the above, there are $\ast$-homomorphisms $E_0,E_1\colon I{\mathcal A}\rightarrow {\mathcal A}$ given by writing $E_0 = e_0 \otimes \id$ and $E_1 = e_1\otimes \id$.

\begin{definition}
Let $\varphi ,\psi \colon {\mathcal A}\dashrightarrow {\mathcal B}$ be asymptotic functors. Then a {\em homotopy} between $\varphi$ and $\psi$ is an asymptotic functor $\theta \colon {\mathcal A}\dashrightarrow I{\mathcal B}$ such that $E_0\circ \theta = \varphi$ and $E_1\circ \theta = \psi$.
\end{definition}

\begin{definition}
Two asymptotic functors $\varphi, \psi \colon \mathcal{A} \dashrightarrow \mathcal{B}$ are called \emph{equivalent} if $\varphi(a)= \psi(a)$ for each object $a \in \Ob ({\mathcal A})$, and for each morphism $x \in \Hom (a,b)_{\mathcal A}$
\[\lim_{t \to \infty} ||\varphi_t(x)- \psi_t(x)|| = 0.\]
\end{definition}

\begin{proposition}\label{Equivalent implies homotopic}
Equivalent asymptotic functors are homotopic. 
\end{proposition}

\begin{proof}
Let $\varphi, \psi \colon \mathcal{A} \dashrightarrow \mathcal{B}$ be equivalent asymptotic functors. Define a homotopy $\theta \colon A \dashrightarrow IB$ by
$\theta_t(a) = \varphi(a) = \psi(a)$ for all $a \in \Ob ({\mathcal A})$ and for all $x \in \Hom (a,b)_{\mathcal A}$, $s \in [0,1]$ by
\[\theta_t(x)(s) = (1-s) \varphi_t(x) + s \psi_t(x) = \varphi_t(x) + s(\psi_t(x) - \varphi_t(x)).\]

It is easy to verify that this is an asymptotic functor, and gives the required homotopy between $\varphi$ and $\psi$.

Then this is an asymptotic functor which defines a homotopy as required. 

\end{proof}

\begin{definition}
Let $\mathcal B$ be a $C^\ast$-category, and let $\mathcal J$ be a $C^\ast$-ideal. A {\em quasicentral set of approximate units} for the pair $({\mathcal B},{\mathcal J})$ consists of a norm-continuous family $\{ u^a_t \ |\ t\in [1,\infty ) \}$ of elements of each space $\Hom (a,a)_{\mathcal J}$ such that:

\begin{itemize}

\item For each $a\in \Ob ({\mathcal B})$ and $t\in [1,\infty )$, the elements $u^a_t$ and $1-u^a_t$ are both positive.

\item $\lim_{t\rightarrow \infty} \| u_t^b x-x \| =0$ for all $x\in \Hom (a,b)_{\mathcal J}$.

\item $\lim_{t\rightarrow \infty} \| u_t^b x-xu_t^a \| =0$ for all $x\in \Hom (a,b)_{\mathcal B}$.

\end{itemize}

\end{definition}

\begin{definition} \label{separable}
We call a $C^\ast$-category $\mathcal A$ {\em separable} if for each $a,b\in \Ob ({\mathcal A})$ we have compact sets $k_n (a,b)$ such that:

\begin{itemize}

\item The union $\displaystyle \cup_{n\in \N} k_n(a,b)$ is dense in $\Hom (a,b)_{\mathcal A}$;

\item For each $n$, and $a,b\in \Ob (\mathcal A)$ we have $k_n(a,b)+k_n(a,b)\subseteq k_{n+1}(a,b)$, $k_n^\ast (a,b) = k_n (a,b)$, and $\lambda k_n (a,b)\subseteq k_{n+1}(a,b)$ whenever $|\lambda |\leq n$;

\item For each $n$, and $a,b,c\in \Ob (\mathcal{A})$, we have $k_n(b,c)k_n(a,b) \subseteq k_{n+1}(a,c)$.

\end{itemize}

\end{definition}

To decode the above notation, we mean:

\begin{itemize}

\item $k_n(a,b) + k_n (a,b) = \{ x+y \ |\ x,y \in k_n (a,b)\}$;

\item $k_n^\ast (a,b) = \{ x^\ast \ |\ x\in k_n (a,b)$;

\item $\lambda k_n (a,b) = \{ \lambda x \ |\ x\in k_n (a,b)$;

\item $k_n(b,c)k_n(a,b) = \{ xy \ |\ x\in K_n (b,c),\ y\in k_n (a,b) \}$.

\end{itemize}

If a $C^\ast$-category  $\mathcal A$ is separable then each Banach space $\Hom (a,b)_{\mathcal A}$ is separable. Conversely, if $A$ is a separable $C^\ast$-algebra, then $A$ is separable in the sense of the above definition.

\begin{lemma}
Let $\mathcal B$ be a separable $C^\ast$-category, and let $\mathcal J$ be a $C^\ast$-ideal. Then the pair $({\mathcal B},{\mathcal J})$ has a quasi-central set of approximate units.
\end{lemma}

\begin{proof}
It is well-known from the theory of $C^\ast$-algebras (see for example \cite{GHT}) that each pair of $C^\ast$-algebras $(\Hom (a,a)_{\mathcal B} , \Hom (a,a)_{\mathcal J})$  has a quasi-central approximate unit $\{ u^a_t \ |\ t\in [1,\infty ) \}$.

Let $x\in \Hom (a,b)_{\mathcal J}$. Then $xx^\ast \in \Hom (b,b)_{\mathcal J}$, and so for all $x\in \Hom (a,b)_{\mathcal J}$, $\lim_{t\rightarrow \infty} \| u_t^b xx^\ast-xx^\ast \| =0$.

By functional calculus, define $v_t^b = (u_t^b)^\frac{1}{2}$. Then again by functional calculus, $u_t^b xx^\ast = (v_t^b x)(v_t^b x)^\ast$ and so, by the $C^\ast$-identity
\[ \lim_{t\rightarrow \infty} \| v_t^b x-x \| =0 . \]

Similarly
\[ \lim_{t\rightarrow \infty} \| xv_t^a x-x \| =0 . \]

By the triangle inequality, we can combine the above two limits to see
\[ \lim_{t\rightarrow \infty} \| v_t^b x -xv_t^a \| =0 \]
so by the $C^\ast$-identity
\[ \lim_{t\rightarrow \infty} \| u_t^b x-xu_t^a \| =0 \]
and we are done.
\end{proof}

Note that if $f\colon [0,1]\rightarrow \F$ is a continuous function, we can define by functional calculus $f(u_t^a)\in \Hom (a,a)_{\mathcal B}$ for each $u_t^a \in \Hom(a,a)_{\mathcal B}$.

\begin{lemma} \label{techexact}
Let $\mathcal B$ be a separable $C^\ast$-category, and let $\mathcal J$ be a $C^\ast$-ideal. Let $(u_t^a)$ be a quasicentral set of approximate units. Let $f\colon [0,1]\rightarrow \F$ be a continuous function such that $f(0)=0$.

\begin{itemize}

\item Let $b\in \Hom (a,b)_{\mathcal B}$. Then
\[ \lim_{t\rightarrow \infty} bf(u_t^a) -f(u_t^b) b =0 . \]

\item If $f(1)=0$ then
\[ \lim_{t\rightarrow \infty} \| f(u_t^b)b \| =0 \]
for all $b\in \Hom (a,b)_{\mathcal J}$.

\end{itemize}

\end{lemma}

\begin{proof}
The set of polynomials with constant coefficient $0$ is dense in the set $\{ f\in C[0,1] \ |\ f(0)=0 \}$ by the Stone-Weierstrass theorem. Hence it suffices to prove the result for such polynomials. To do this, it suffices to prove the result for the polynomial $f(x)=x$. The first part of the result therefore follows from the definition of quasicentral set of approximate units.

As for the second part of the result, observe that the set of functions $f\in C[0,1]$ such that $\lim_{t\rightarrow \infty} \| f(u_t^b)b \| =0$
for all $b\in \Hom (a,b)_{\mathcal J}$ is an ideal in the $C^\ast$-algebra of functions $\{ f\in C[0,1] \ |\ f(0)=f(1)=1 \}$. By the Stone-Weierstrass theorem, as above, this algebra is generated by the function $f(x)=x(1-x)$, so it suffices to prove the result for this function. By the definition of a quasicentral set of approximate units
\[ \lim_{t\rightarrow \infty}  \| f(u_t^b)b \| = \lim_{t\rightarrow \infty}  \| u_t^b (1-ut^b)b \| \leq \lim_{t\rightarrow \infty}  \| (1-ut^b)b \| =0 \]
for all $b\in \Hom (a,b)_{\mathcal J}$.
\end{proof}

%
%
%

\begin{definition}
Let ${\mathcal A}$ be a $C^\ast$-category.  Then we define the
{\em additive completion}, ${\mathcal A}_\oplus$, to be the category
in which the objects are formal sequences of the form
$$a_1 \oplus \cdots \oplus a_n \qquad a_i \in \Ob ({\mathcal A})$$
Repetitions are allowed in such formal sequences.  The empty
sequence is also allowed, and labelled $0$.
The morphism set $\Hom (a_1 \oplus \cdots \oplus a_m , b_1 \oplus
\cdots \oplus b_n )$ is defined to be the set of matrices of the
form
$$\left( \begin{array}{ccc}
x_{1,1} & \cdots & x_{1,m} \\
\vdots & \ddots & \vdots \\
x_{n,1} & \cdots & x_{n,m} \\
\end{array} \right)  \qquad x_{i,j} \in \Hom (a_j , b_i )$$
and composition of morphisms is defined by matrix multiplication.
The involution is defined by the formula
$$\left( \begin{array}{ccc}
x_{1,1} & \cdots & x_{1,m} \\
\vdots & \ddots & \vdots \\
x_{n,1} & \cdots & x_{n,m} \\
\end{array} \right)^\ast 
=
\left( \begin{array}{ccc}
x_{1,1}^\ast & \cdots & x_{n,1}^\ast \\
\vdots & \ddots & \vdots \\
x_{1,m}^\ast & \cdots & x_{n,m}^\ast \\
\end{array} \right)^\ast   \qquad x_{i,j} \in \Hom (a_j , b_i ) .$$
\end{definition}
Given a $\ast$- functor $\alpha \colon {\mathcal A}
\rightarrow {\mathcal B}$, there is an induced $\ast$-functor
$\alpha_\oplus \colon {\mathcal A}_\oplus \rightarrow {\mathcal
B}_\oplus$ defined by writing
$$\alpha_\oplus (a_1 \oplus \cdots a_n ) = \alpha (a_1) \oplus \cdots \oplus \alpha (a_n) \qquad a_i \in \Ob ({\mathcal A})$$
and
$$\alpha_\oplus \left( \begin{array}{ccc}
x_{1,1} & \cdots & x_{1,m} \\
\vdots & \ddots & \vdots \\
x_{n,1} & \cdots & x_{n,m} \\
\end{array} \right) =
\left( \begin{array}{ccc}
\alpha (x_{1,1}) & \cdots & \alpha (x_{1,m}) \\
\vdots & \ddots & \vdots \\
\alpha (x_{n,1}) & \cdots & \alpha (x_{n,m}) \\
\end{array} \right)  \qquad x_{i,j} \in \Hom (a_j , b_i )$$
If we have a faithful representation $\rho \colon {\mathcal A}\rightarrow {\mathcal L}$, this therefore extends to a faithful representation $\rho_\oplus \colon {\mathcal A}_\oplus \rightarrow {\mathcal L}$. We now use the same trick as we did in the spatial tensor product to define the norm on the morphism sets of the additive completion ${\mathcal A}_\oplus$.

\begin{definition}
Let $\mathcal A$ and $\mathcal B$ be $C^\ast$-categories. We call a $\ast$-functor $\alpha \colon {\mathcal A}_\oplus \rightarrow {\mathcal B}_\oplus$ {\em additive} if:

\begin{itemize}

\item $\alpha (a_1\oplus a_2 ) = \alpha (a_1) \oplus \alpha (a_2)$ for all objects $a_1,a_2\in \Ob ({\mathcal A})$.

\item 
\[ \alpha_\oplus \left( \begin{array}{ccc}
x_{1,1} & \cdots & x_{1,m} \\
\vdots & \ddots & \vdots \\
x_{n,1} & \cdots & x_{n,m} \\
\end{array} \right) =
\left( \begin{array}{ccc}
\alpha (x_{1,1}) & \cdots & \alpha (x_{1,m}) \\
\vdots & \ddots & \vdots \\
\alpha (x_{n,1}) & \cdots & \alpha (x_{n,m}) \\
\end{array} \right) \]
for all $x_{i,j} \in \Hom (a_j , b_i )$.

\end{itemize}

\end{definition}

Given a $\ast$-functor $\alpha \colon {\mathcal A}\rightarrow {\mathcal B}$, the induced $\ast$-functor $\alpha_\oplus \colon {\mathcal A}_\oplus \rightarrow {\mathcal B}_\oplus$ is clearly additive. If we have a $\ast$-functor $\alpha \colon {\mathcal A}\rightarrow {\mathcal B}_\oplus$, we can also define an additive  $\ast$-functor $\alpha_\oplus \colon {\mathcal A}_\oplus \rightarrow {\mathcal B}_\oplus$ by the same construction as above.

It is easy to check that with the above induced additive $\ast$-functors, the the assignment ${\mathcal A}\mapsto {\mathcal A}_\oplus$ is a functor from the category of $C^\ast$-categories and $\ast$-functors to the category of $C^\ast$-categories and additive $\ast$-functors.

We can proceed in much the same way with asymptotic functors. To be precise, let $\varphi \colon {\mathcal A}\dashrightarrow {\mathcal B}_\oplus$ be an asymptotic functor. Then we have an asympotic functor $\varphi_\oplus \colon {\mathcal A}_\oplus \dashrightarrow {\mathcal B}_\oplus$ by writing

\begin{itemize}

\item $\varphi (a_1\oplus a_2 ) = \varphi (a_1) \oplus \varphi (a_2)$ for all objects $a_1,a_2\in \Ob ({\mathcal A})$.

\item 
\[ (\varphi_\oplus )_t \left( \begin{array}{ccc}
x_{1,1} & \cdots & x_{1,m} \\
\vdots & \ddots & \vdots \\
x_{n,1} & \cdots & x_{n,m} \\
\end{array} \right) =
\left( \begin{array}{ccc}
\varphi_t (x_{1,1}) & \cdots & \varphi_t (x_{1,m}) \\
\vdots & \ddots & \vdots \\
\varphi_t (x_{n,1}) & \cdots & \varphi_t (x_{n,m}) \\
\end{array} \right) \]
for all $x_{i,j} \in \Hom (a_j , b_i )$ and $t\in [1,\infty)$.

\end{itemize}

The induced asymptotic morphism $\varphi_\oplus$ is additive in the same sense as the above.

\begin{lemma}
Let \[ 0 \rightarrow {\mathcal J}\stackrel{i}{\rightarrow} {\mathcal B}\stackrel{j}{\rightarrow} {\mathcal A}\rightarrow 0 \] be a short exact sequence of graded $C^\ast$-categories. Then the induced sequence
\[ 0 \rightarrow {\mathcal J}_{\oplus}\stackrel{i_{\oplus}}{\rightarrow} {\mathcal B}_{\oplus}\stackrel{j_{\oplus}}{\rightarrow} {\mathcal A_{\oplus}}\rightarrow 0 \]
is also short exact.
\end{lemma}

\begin{proof}
Let $a_i$ and $b_j$ be objects. Since the $\ast$-functor $i$ is injective at the level of objects, and 
$$i_{\oplus} \left( \begin{array}{ccc}
x_{1,1} & \cdots & x_{1,m} \\
\vdots & \ddots & \vdots \\
x_{n,1} & \cdots & x_{n,m} \\
\end{array} \right) =
\left( \begin{array}{ccc}
i (x_{1,1}) & \cdots & i (x_{1,m}) \\
\vdots & \ddots & \vdots \\
i (x_{n,1}) & \cdots & i (x_{n,m}) \\
\end{array} \right)  \qquad x_{i,j} \in \Hom (a_j , b_i )$$
it follows that $i_{\oplus}$ is also injective. Similarly, the $\ast$-functor $j_{\oplus}$ is surjective.

So it suffices to check that $\text{Ker} ~ j_{\oplus} = \text{Im} ~ i_{\oplus}$. Since $\text{Ker} ~ j = \text{Im} ~ i$, we have agreement in each term of the matrices. and the result follows. 
\end{proof}

\section{Composition}

Let $\mathcal A$, $\mathcal B$, and $\mathcal C$ be $C^\ast$-categories, and let $\varphi \colon {\mathcal A}\dashrightarrow {\mathcal B}$ and $\psi \colon {\mathcal B}\dashrightarrow {\mathcal C}$ be asymptotic functors. We would like to be able to define the composition $\psi \circ \varphi \colon {\mathcal A}\dashrightarrow {\mathcal C}$. Unfortunately, we cannot simply define $(\varphi \circ \psi )_t = \varphi_t \circ \psi_t$ for each value $t$ and expect an asymptotic morphism; this fails even in the $C^\ast$-algebra case.

Fortunately, there are cases where the obvious construction works, and a reparametrisation procedure which works more generally.

\begin{proposition}
Let $\alpha \colon {\mathcal A}\rightarrow {\mathcal B}$ be a $\ast$-functor, and let $\varphi \colon {\mathcal B}\rightarrow {\mathcal C}$ be an asymptotic functor. Then we have an asymptotic functor $\varphi \circ \alpha \colon {\mathcal A}\rightarrow {\mathcal C}$ defined by writing $(\varphi \circ \alpha )_t = \varphi_t \circ \alpha$ where $t\in [1,\infty )$.
\end{proposition}

\begin{proof}
As noted in \cite{Mitch2}, the $\ast$-functor $\alpha$ is norm-decreasing on each morphism set, as well as being linear, and compatible with the involution and composition in $\mathcal B$ and $\mathcal C$. The required properties of an asymptotic morphism are now easy to check.
\end{proof}

The following is similar.

\begin{proposition}
Let $\varphi \colon {\mathcal A}\dashrightarrow {\mathcal B}$ be an asymptotic functor, and let $\alpha \colon {\mathcal B}\rightarrow {\mathcal C}$ be a $\ast$-functor. Then we have an asymptotic functor $\alpha \circ \varphi \colon {\mathcal A}\dashrightarrow {\mathcal C}$ defined by writing $(\alpha \circ \varphi )_t = \alpha \circ \varphi_t$ where $t\in [1,\infty )$.
\end{proposition}

We have set things up with definition \ref{separable} so that the proof of the following result is the same as the proof of theorem 25.3.1 of \cite{Black}.

\begin{theorem} \label{composition}
Let $\mathcal A$, $\mathcal B$, and $\mathcal C$ be separable $C^\ast$-categories. Let $\varphi \colon {\mathcal A}\dashrightarrow {\mathcal B}$ and $\psi \colon {\mathcal B}\dashrightarrow {\mathcal C}$ be asymptotic functors. Then there is a continuous increasing function $r\colon [1,\infty )\rightarrow [1,\infty )$ such that we have an asymptotic functor $\psi \circ_r \varphi \colon {\mathcal B}\dashrightarrow {\mathcal C}$ defined by the formula $(\psi \circ_r \varphi )_t = \psi_{r(t)} \circ \varphi_t$.

Further:

\begin{itemize}

\item If we have continuous functions $r,s\colon [1,\infty )\rightarrow [1,\infty )$ such that $\psi \circ_r \varphi$ is an asymptotic functor, and $s(t)\geq r(t)$ for all $t$, then $\psi\circ_s \varphi$ is an asymptotic functor.

\item The homotopy class of the asymptotic functor $\psi \circ_r \varphi$ is independent of $r$, and depends only on the homotopy classes of the asymptotic functors $\phi$ and $\varphi$. 

\end{itemize}

\end{theorem}

\section{Graded $E$-theory}
Here we give the definition of gradings and generalise the framework already set out to $C^\ast$-categories with gradings. Then we prove we have an abelian group structure on the homotopy class of asymptotic functors from $\mathcal{S} \widehat{\otimes} \mathcal A_{\oplus}$ to $B_\oplus \widehat{\otimes} \mathcal{K}$.

\begin{definition}
Let $\mathcal{A}$ be a $C^\ast$-category. A \emph{grading} on $\mathcal{A}$ is a $\ast$-functor $\delta_{\mathcal{A}} \colon \mathcal{A} \rightarrow \mathcal{A}$ such that for each object $a \in \mathcal{A}$, $\delta_{\mathcal{A}} (a) = a$, and ${\delta}_{\mathcal{A}}^2 = \text{id}$.
\end{definition}

For each morphism set $\Hom (a,b)_{\mathcal A}$ in a graded $C^\ast$-category we have a decomposition
\[ \Hom (a,b)_{\mathcal A} = \Hom (a,b)_{\mathcal A}^{\text{even}} \oplus \Hom (a,b)_{\mathcal A}^{\text{odd}}.\] 
where $x\in \Hom (a,b)_{\mathcal A}^{\text{even}}$ if $\delta_{\mathcal A} (x)=x$, and $x\in \Hom (a,b)_{\mathcal A}^{\text{odd}}$ if $\delta_{\mathcal A} (x)=-x$. We define the {\em degree}, $\deg (x)$, of a morphism $x$  to be $0$ if $x\in \Hom (a,b)_{\mathcal A}^{\text{even}}$ and $1$ if $x\in \Hom (a,b)_{\mathcal A}^{\text{odd}}$.

A $C^\ast$-category equipped with a grading is called a {\em graded} $C^\ast$-category. We can also consider graded $\ast$-functors and graded asymptotic functors.

A graded $C^\ast$-algebra can be considered the same thing as a graded $C^\ast$-category with one object. In our constructions, we make particular use of the following graded $C^\ast$-algebra.

\begin{definition}
We define $S$ to be the $C^\ast$-algebra $C_0(\R )$ equipped with the grading defined by saying an element $f\in C_0 (X)$ has degree $0$ if the function $f$ is even, and degree $1$ if the function $f$ is odd.
\end{definition}

We can realise the grading on $S$ by the $\ast$-homomorphism $\delta_S\colon S\rightarrow S$ defined by the formula $\delta (f)(x) = f(-x)$. We distinguish the graded $C^\ast$-algebra $S$ from $C_0 (\R )$, which as a graded $C^\ast$-algebra has the {\em trivial} grading, where ever element has degree $0$.

\begin{definition}
Let $\mathcal{A}$ and $\mathcal{B}$ be graded $C^\ast$-categories with gradings $\delta_{\mathcal{A}}$ and $\delta_{\mathcal{B}}$ respectively. 
A \emph{graded $\ast$-functor} $\alpha \colon \mathcal{A} \rightarrow \mathcal{B}$ is a $\ast$-functor such that for all $x \in \Hom (a,b)_{\mathcal A}$ we have $\delta_{\mathcal B} (\alpha (x)) = \alpha (\delta_{\mathcal A}(x))$.

A \emph{graded asymptotic functor} $\varphi = \varphi_t \colon \mathcal{A} \dashrightarrow \mathcal{B}$ is an asymptotic functor with the additional property that for all $x \in \Hom (a,b)_{\mathcal A}$ we have
\[ \lim_{t\rightarrow \infty} \| \delta_{\mathcal{B}}(\varphi_t (x)) -\varphi_t (\delta_{\mathcal{A}}(x)) \| = 0. \]
\end{definition}

For a morphism $x\in \Hom (a,b)_{\mathcal A}$ let us define the {\em degree} of $x$ by writing $\deg (x)=0$ if $x\in \Hom (a,b)_{\mathcal A}^{\text{even}}$, and $\deg (x)=1$ if $x\in \Hom (a,b)_{\mathcal A}^{\text{odd}}$. We call morphisms of degree $0$ even, and morphisms of degree $1$ odd.

\begin{definition}
Let $\mathcal A$ and $\mathcal B$ be graded $C^\ast$-categories, with gradings $\delta_{\mathcal{A}}$ and $\delta_{\mathcal{B}}$ respectively. Then we define the {\em graded tensor product} ${\mathcal A}\gtp {\mathcal B}$ to be the same collection of objects and Banach spaces of morphisms as the spatial tensor product ${\mathcal A}\otimes {\mathcal B}$, but with involution and composition law defined by the formulae
\begin{itemize}
\item $(x\otimes y)^\ast = (-1)^{\deg (x)\deg (y)}x^\ast \otimes y^\ast$;
\item $(u\otimes v)(x\otimes y) = (-1)^{\deg (v)\deg (x)} ux\otimes vy$,
\end{itemize}
and extending by linearity.
\end{definition}

With the above alternative composition law, the graded tensor product ${\mathcal A}\gtp {\mathcal B}$ is a $C^\ast$-category. We have a grading, $\delta$, defined by the formula
\[ \delta (x\otimes y) = \delta_{\mathcal A} (x)\otimes \delta_{\mathcal B}(y) \]
on elementary tensors; we again extend by linearity.

Let us equip $C[0,1]$ with the trivial grading, where every morphism is even. Then if $\mathcal A$ is a graded $C^\ast$-category, we can form the tensor product $I{\mathcal A} = {\mathcal A}\gtp C[0,1]$. as in the previous section, we have graded asymptotic $\ast$-functors $E_0,E_1 \colon I{\mathcal A}\rightarrow {\mathcal A}$, and we can form a notion of homotopy.

\begin{definition}
Let $\varphi ,\psi \colon {\mathcal A}\dashrightarrow {\mathcal B}$ be graded asymptotic functors. Then a {\em homotopy} between $\varphi$ and $\psi$ is a graded asymptotic functor $\theta \colon {\mathcal A}\dashrightarrow I{\mathcal B}$ such that $E_0\circ \theta = \varphi$ and $E_1\circ \theta = \psi$.
\end{definition}

Looking at objects, observe that
\begin{itemize}
\item $\theta_t(a) =\varphi_t(a)=\psi_t (a)$ for each object $a\in \Ob ({\mathcal A})$.
\item Given $x\in \Hom (a,b)_{\mathcal A}$. we have $\theta_t (x)(0) = \varphi_t (x)$ and $\theta_t (x)(1) = \psi_t (x)$ for all $t\in [1,\infty )$.
\end{itemize}

Being homotopic is an equivalence relation on the set of all graded asymptotic functors between graded $C^\ast$-categories $\mathcal A$ and $\mathcal B$. We denote the set of equivalence classes by $\llbracket {\mathcal A} , {\mathcal B} \rrbracket$. We also write $\llbracket {\mathcal A}_\oplus , {\mathcal B}_\oplus \rrbracket_\oplus$ to denote the set of homotopy equivalence classes of additive graded asymptotic functors. The obvious map
\[ \llbracket {\mathcal A} , {\mathcal B}_\oplus \rrbracket \rightarrow  \llbracket {\mathcal A}_\oplus , {\mathcal B}_\oplus \rrbracket_\oplus  \]
is a one to one correspondence.

We call a Hilbert space $H$ {\em graded} if we have a decomposition $H=H_0\oplus H_1$. If $H$ and $H'$ are graded Hilbert spaces, and $T\colon H\rightarrow H'$ is a bounded linear map, we call $T$ {\em even} if $T[H_0]\subseteq H'_0$ and $T[H_1]\subseteq H'_1$, and {\em odd} if $T[H_0]\subseteq H'_1$ and $T[H_1]\subseteq H'_0$. The odd and even bounded linear maps give us a decomposition
\[ \Hom (H,H')_{\mathcal L} = \Hom (H, H')_{\mathcal L}^{\text{even}} \oplus  \Hom (H, H')_{\mathcal L}^{\text{odd}} \]
on each morphism set on the $C^\ast$-category $\mathcal L$ made up of all graded separable Hilbert spaces and bounded linear maps. Hence we view $\mathcal L$ as a graded $C^\ast$-category. We have a non-unital graded $C^\ast$-subcategory $\mathcal K$ of all compact operators between graded separable Hilbert spaces.

Giiven  a graded Hilbert space $H=H_0\oplus H_1$, let $H^\mathrm{opp}$ be the Hilbert space we obtain by reversing the grading, that is to say $H=H_1\oplus H_0$.

Observe that we have natural graded $\ast$-isomorphisms
\[\mathcal{K}\oplus \mathcal{K} \rightarrow \mathcal{K} \qquad {\mathcal K}\gtp {\mathcal K}\rightarrow {\mathcal K} \]

Observe that ${\mathcal K}_\oplus = {\mathcal K}$, and so for any $C^\ast$-category $\mathcal A$ we have
\[ ({\mathcal A}\gtp {\mathcal K})_\oplus = {\mathcal A}_\oplus \gtp {\mathcal K} = {\mathcal A}\gtp {\mathcal K} \]

Now, observe that the algebra $S$ is generated by two functions $u,v\in S$ described by the formulae $u(x) = e^{-x^2}$ and $v(x) = x e^{-x^2}$ respectively. We define a graded $\ast$-homomorphism $\Delta \colon S\rightarrow S\gtp S$ by writing $\Delta (u) =u\otimes u$ and $\Delta (v) = u\otimes v + v\otimes u$. As shown in section 1.3 of 
\cite{GH04}, $\Delta \colon S\rightarrow S\gtp S$ is the unique $\ast$-homomorphism with the property
\[ \Delta (f) = f\otimes 1 + 1\otimes f \]
whenever $f$ has compact support.

\begin{lemma}
We have an abelian semigroup operation on the set $\llbracket \mathcal{A}_\oplus, {\mathcal B}_\oplus \gtp {\mathcal K} \rrbracket_\oplus$ defined by the formula $[(\varphi_t)]+[(\psi_t)] =[(\varphi_t) \oplus (\psi_t) ]$. Further, the abelian semigroup we define has an identity defined by taking the class of the zero asumptotic functor.
\end{lemma}

\begin{proof}
We can take the direct sum of two additive graded asymptotic functors $\mathcal{A}_\oplus\dashrightarrow {\mathcal B}_\oplus \gtp {\mathcal K}$ in the obvious way. The resulting operation, defined above, is certainly well-defined and associative, and the the class of the zero asymptotic functor, mapping every object of ${\mathcal A}_\oplus$ to the zero object of ${\mathcal B}_\oplus$  is the identity.

So we must verify commutativity. In other words, given additive graded asymptotic morphisms $\varphi_t, \psi_t \colon {\mathcal A}_\oplus \dashrightarrow {\mathcal B}_\oplus \gtp {\mathcal K}$, we want to define an additive graded asymptotic homotopy between $\varphi_t \oplus \psi_t$ and $\psi_t\oplus \varphi_t$.

Let
\[ R_\theta = \left( \begin{array}{cc}
\cos (\pi \theta /2) & -\sin (\pi \theta /2) \\
\sin (\pi \theta /2) &  \cos (\pi \theta /2) \\
\end{array} \right)  \qquad \theta \in [0,1] \]

Then $R_\theta$ is a morphism in the category $\mathcal K$, so for elements $x,y\in \Hom (a,b)_{{\mathcal B})_\oplus}\otimes {\mathcal K}$ we can form the element $R_\theta (x\oplus y)R_\theta^\ast$.

Now $R_0 (x\oplus y)R_0^\ast = x\oplus y$ and 
\[ R_1 (x\oplus y)R_1^\ast = \left( 
\begin{array}{cc} 0 & -1 \\
1 & 0 \\
\end{array} \right) \left( \begin{array}{cc} x & 0 \\
0 & y \\
\end{array} \right) \left( \begin{array}{cc} 0 & 1 \\
-1 & 0 \\
\end{array} \right)
= y\oplus x \]
The
Hence, more generally we have a homotopy between asymptotic functors $(\varphi_t) \oplus (\psi_t)$ and $(\psi_t)\oplus (\varphi_t)$ which makes the semigroup $\llbracket \mathcal{A}_\oplus, {\mathcal B}_\oplus \gtp {\mathcal K} \rrbracket_\oplus$ commutative
\end{proof}

\begin{lemma}
The set $\llbracket {S} \widehat{\otimes} \mathcal A_{\oplus}  ,  B_\oplus \widehat{\otimes} \mathcal{K} \rrbracket_\oplus$ is an abelian group under the group operation defined above. 
\end{lemma}

\begin{proof}
By the above it suffices to check that we have inverses.

Let $\varphi_t \colon {S} \gtp \mathcal{A}_\oplus \dashrightarrow  \mathcal{B}_\oplus \gtp \mathcal{K}$ be a graded asymptotic additive functor, and let $\delta \colon S\gtp {\mathcal  A}_\oplus \rightarrow S\gtp {\mathcal  A}_\oplus$ be the grading, defined by taking the tensor product of the gradings on $S$ and on ${\mathcal A}_\oplus$. Define $\varphi_t^{\text{opp}} = \varphi_t \circ \delta$.

Let $s\geq 0$ be a fixed scalar, and define $\Phi_t^s \colon S\gtp S\gtp {\mathcal A}_\oplus\gtp {\mathcal H}\dashrightarrow {\mathcal B}_\oplus \gtp {\mathcal K}$ by:
\[ \Phi_t^s (a) = \varphi_t (a) \qquad a\in \Ob ({\mathcal A}_\oplus ) \]
and
\[ \Phi_t^s (f\otimes x) = f(x)(\varphi_t (x)\oplus \varphi_t^{\text{opp}}{(x)}) \qquad f\in S,\ x\in \Hom (a,b)_{S\gtp {\mathcal A}_\oplus} \]

Then $(\Phi_t^s)_{t\in [1,\infty )}$ is a graded asymptotic functor. Moreover, as $s\rightarrow \infty$ we have that $\Phi_t^s\rightarrow 0$. Observe
\[ \Phi_t^s (u\otimes x) = \varphi_t (x)\oplus \varphi_t^{\text{opp}}{(x)} \qquad \Phi_t^s (v\otimes x) = 0 \]
for the functions  $u(x) = e^{-x^2}$ and $v(x) = x e^{-x^2}$ which generate $S$.

The function $s\mapsto s/(1-s)$ is a monotone increasing bijection $[0,1)\rightarrow [0,\infty)$.  We have a graded $\ast$-homomorphism $\Delta \colon S\rightarrow S\gtp S$ defined by writing $\Delta (u)=u\otimes u$ and $\Delta (v) = u\otimes v+v\otimes u$. So we can define $\theta \colon {S} \gtp \mathcal{A}_\oplus\dashrightarrow I\mathcal{B}_\oplus \gtp \mathcal{K}$ by the formula.
\[ \theta (x ) = \Phi_t^{s/(1-s)} \Delta \gtp \text{id}_{A \widehat{\otimes} \mathcal{K}(\mathcal{H})} (x) \]

Observe
\begin{align*}
\Phi_t^{s/(1-s)}(\Delta \widehat{\otimes} \text{id}_{A \widehat{\otimes} \mathcal{K}(\mathcal{H})}) (u \widehat{\otimes} x) 
& = \Phi_t^{s/(1-s)}(\Delta(u) \widehat{\otimes} x) \\
& = \Phi_t^{s/(1-s)}( u  \widehat{\otimes} u \widehat{\otimes} x) \\
& = u(s) (\varphi_t \oplus \varphi_t^{\text{opp}})(u \widehat{\otimes} x),  
\end{align*}
and when $s=0$, $u(s) = e^{-0} = 1$ and so we obtain $\varphi_t \oplus \varphi_t^{\text{opp}}$ and when $s\rightarrow 1$, we obtain $0$. 

Now for $v$, 
\begin{align*}
\Phi_t^{s/(1-s)}(\Delta \widehat{\otimes} \text{id}_{A \widehat{\otimes} \mathcal{K}(\mathcal{H})}) (v \widehat{\otimes} x) 
& = \Phi_t^{s/(1-s)}(\Delta(v) \widehat{\otimes} x) \\
& = \Phi_t^{s/(1-s)}((u \widehat{\otimes} v + v \widehat{\otimes} u) \widehat{\otimes} x) \\ 
&  \sim \Phi_t^{s/(1-s)}((u \widehat{\otimes} v) \widehat{\otimes} x) + \Phi_t^s((v \widehat{\otimes} u) \widehat{\otimes} x) \\ 
& = u(s) (\varphi_t \oplus \varphi_t^{\text{opp}})(v \widehat{\otimes} x) +  v(s) (\varphi_t \oplus \varphi_t^{\text{opp}})(u \widehat{\otimes} x) ,\\ 
\end{align*}
where the equivalence is valid, since equivalent graded asymptotic morphisms are homotopic by Proposition~\ref{Equivalent implies homotopic}. Now at $s=0$, $v(s) = 0$ and as $s \to 1$, $v(s) \to 0$, so the second term is equal to $0$ at the end points, and hence we obtain the same endpoints as above and we are done. 
\end{proof}

\begin{definition}
For $C^\ast$-categories $\mathcal{A}$ and $\mathcal{B}$, we define the graded $E$-theory groups by 
\[E^n(\mathcal{A}, \mathcal{B}) = \llbracket S \widehat{\otimes} \mathcal{A}_\oplus \widehat{\otimes} \mathcal{K} , \Sigma^n \mathcal{B}_\oplus \widehat{\otimes} \mathcal{K} \rrbracket_\oplus \]
\end{definition}

As a special case we write
\[ E({\mathcal A},{\mathcal B}) = E^0({\mathcal A},{\mathcal B}) =  \llbracket S \widehat{\otimes} \mathcal{A}_\oplus \widehat{\otimes} \mathcal{K} , \mathcal{B}_\oplus \widehat{\otimes} \mathcal{K} \rrbracket_\oplus .\]

We can define a graded $\ast$-functor $c\colon S \rightarrow \F$ by evaluating a function at zero. Given an asympotic functor $\varphi \colon \mathcal{A}\dashrightarrow \mathcal{B}$ we have an equivalence class $[\alpha ] = [c \otimes \alpha_\oplus \otimes \id ] \in E^0 ({\mathcal A},{\mathcal B})$.

\section{The Product and Functoriality}

To explore the other properties of $E$-theory, we first set up an associative product. Let $\varphi \colon S\gtp {\mathcal A}_{\oplus}\gtp {\mathcal K} \dashrightarrow {\mathcal B}_{\oplus} \gtp {\mathcal K}$ and  $\psi \colon S\gtp {\mathcal B}_{\oplus} \gtp {\mathcal K} \dashrightarrow {\mathcal C}_{\oplus}\gtp {\mathcal K}$ be graded asymptotic functors. We have a canonical graded $\ast$-homomorphism $\Delta \colon S\rightarrow S\gtp S$. So by theorem \ref{composition} we have an increasing map $r\colon [1,\infty )\rightarrow [1,\infty )$ giving us an asymptotic functor $\psi \circ_r \varphi \colon S\gtp {\mathcal A}_{\oplus} \gtp {\mathcal K} \rightarrow {\mathcal C}_{\oplus}\gtp {\mathcal K}$
\[ S\gtp {\mathcal A}_{\oplus}\gtp {\mathcal K} \stackrel{\Delta \otimes \mathrm{id}}{\longrightarrow}
S\gtp S\gtp {\mathcal A}_{\oplus}\gtp {\mathcal K} \stackrel{ \mathrm{id} \otimes \varphi_t}{\longrightarrow}
S\gtp {\mathcal B}_{\oplus}\gtp {\mathcal K} \stackrel{ \psi_{r(t)}}{\longrightarrow}
{\mathcal C}_{\oplus}\gtp {\mathcal K} \]

Further, by theorem \ref{composition} the homotopy class of the asymptotic morphism $\psi \circ_r \varphi$ depends only on the homotopy classes $\psi$ and $\varphi$ and not on the map $r$, so we have a well-defined map of $E$-theory groups
\[ E({\mathcal A},{\mathcal B}) \times E({\mathcal B},{\mathcal C}) \rightarrow E({\mathcal A},{\mathcal C}) \]
given by the formula
\[ ([\varphi ] , [\psi ] ) \mapsto [\psi \circ_r \varphi ] \]

Note also that composition of maps is associative, so the above product is associative. To be precise, this means that the products
\[ (E({\mathcal A},{\mathcal B}) \times E({\mathcal B},{\mathcal C}) )\times E({\mathcal C},{\mathcal D}) \rightarrow  E({\mathcal A},{\mathcal C}) \times E({\mathcal C},{\mathcal D}) \rightarrow  E({\mathcal A},{\mathcal D}) \]
and
\[ E({\mathcal A},{\mathcal B}) \times ( E({\mathcal B},{\mathcal C}) \times E({\mathcal C},{\mathcal D})) \rightarrow  E({\mathcal A},{\mathcal B}) \times E({\mathcal B},{\mathcal D}) \rightarrow  E({\mathcal A},{\mathcal D}) \]
are equal.

If $\mathcal A$ and $\mathcal B$ are categories, a {\em bivariant functor}, $F$, from $\mathcal A$ to $\mathcal B$ is a functor $F\colon {\mathcal A}^\mathrm{op}\times {\mathcal A} \rightarrow {\mathcal B}$. In other words, for objects $a,b\in \Ob ({\mathcal A})$ we obtain an object $F(a,b)\in \Ob ({\mathcal B})$, which is a contravariant functor in the first variable and a covariant functor in the second variable.

\begin{lemma}[Functoriality]
$E$ is a bivariant functor from the category where objects are graded $C^\ast$-categories and morphisms are $\ast$-functors to the category where objects are abelian groups and morphisms are group homomorphisms. 
\end{lemma}

\begin{proof}
The identity property required for functors is clearly satisfied.

Let $\mathcal{A}, \mathcal{B}, \mathcal{C}, \mathcal{D}$ be graded $C^\ast$-categories. Let $\alpha \colon \mathcal{A} \rightarrow \mathcal{B}$ be a graded $\ast$-functor, then we have an abelian group $E(\mathcal{A},\mathcal{D}) = \llbracket S \widehat{\otimes} \mathcal{A}_{\oplus} \widehat{\otimes} \mathcal{K}, \mathcal{D}_{\oplus} \widehat{\otimes} \mathcal{K} \rrbracket_{\oplus}$ for all $\mathcal{A}$ and a group homomorphism $\alpha^{\ast}\colon E(\mathcal{B},\mathcal{D}) \rightarrow E(\mathcal{A},\mathcal{D})$ defined by $\alpha^{\ast}(\llbracket x\rrbracket) =\llbracket x . \alpha\rrbracket$, where $(x . \alpha)_t= x_t \circ (\text{id}_S \widehat{\otimes}\alpha_{\oplus}  \widehat{\otimes} \text{id}_{\mathcal{K}})$ for all $\llbracket x \rrbracket \in E(\mathcal{B},\mathcal{D})$. 

Now consider the composition of graded $\ast$-functors $\mathcal{A} \xrightarrow{\alpha} \mathcal{B} \xrightarrow{\beta} \mathcal{C}$ on a representative $x$ of $\llbracket x \rrbracket$, 
\begin{align*}
(\beta \circ \alpha)^{\ast}(x_t)
& = x_t \circ (\text{id}_{\mathcal{S}} \widehat{\otimes}(\beta_{\oplus} \circ \alpha_{\oplus}) \widehat{\otimes} \text{id}_{\mathcal{K}}) \\
& = x_t \circ (\text{id}_{\mathcal{S}} \widehat{\otimes} \beta_{\oplus}  \widehat{\otimes} \text{id}_{\mathcal{K}}) \circ (\text{id}_{\mathcal{S}} \widehat{\otimes} \alpha_{\oplus} \widehat{\otimes} \text{id}_{\mathcal{K}}) \\
& = \beta^{\ast}(x_t) \circ (\text{id}_{\mathcal{S}} \widehat{\otimes} \alpha_{\oplus} \widehat{\otimes} \text{id}_{\mathcal{K}}) \\
& = \alpha^{\ast}\beta^{\ast}(x_t) \\
& = (\alpha^\ast \circ \beta^{\ast})(x_t). \\
\end{align*}

Similarly, we have for each graded $C^\ast$-category $\mathcal{A}$, an abelian group $E(\mathcal{D},\mathcal{A}) =  \llbracket  S \widehat{\otimes} \mathcal{D}_{\oplus}\widehat{\otimes} \mathcal{K}, \mathcal{A}_{\oplus}\widehat{\otimes} \mathcal{K} \rrbracket_{\oplus}$ and for a graded $\ast$-functor $\alpha$, a group homomorphism $\alpha_{\ast} \colon E(\mathcal{D}, \mathcal{A}) \rightarrow E(\mathcal{D}, \mathcal{B})$ defined by $\alpha_{\ast}(\llbracket y\rrbracket) = \llbracket \alpha . y\rrbracket$, where $(\alpha . y)_t =(\alpha_{\oplus} \widehat{\otimes} \id_{\mathcal{K}} ) \circ y_t$ for all $\llbracket y \rrbracket \in E(\mathcal{D},\mathcal{A})$. Considering the composition of morphisms above and taking a representative $y$ of $\llbracket y \rrbracket \in E(\mathcal{D},\mathcal{A})$ we see that 
\begin{align*}
(\beta \circ \alpha)_{\ast}(y_t) 
& = ((\beta_{\oplus} \circ \alpha_{\oplus})  \widehat{\otimes}  \text{id}_{\mathcal{K}}) \circ y_t \\
& = ((\beta_{\oplus} \circ \alpha_{\oplus})  \widehat{\otimes} \text{id}_{\mathcal{K}}) \circ y_t \\
& = (\beta_{\oplus}  \widehat{\otimes}  \text{id}_{\mathcal{K}}) \circ (\alpha_{\oplus} \widehat{\otimes}  \text{id}_{\mathcal{K}}) \circ y_t \\
& = (\beta_{\oplus} \widehat{\otimes}  \text{id}_{\mathcal{K}}) \circ \alpha_{\ast}(y_t) \\
& = \beta_{\ast}\alpha_{\ast}(y_t) \\
& = (\beta_{\ast} \circ \alpha_{\ast})(y_t). \\
\end{align*}
\end{proof}

Note that the above functorially induced maps are a special case of the product. It follows that the product is a natural map, that is to say it commutes with the homomorphisms $\alpha^\ast$ and $\alpha_\ast$ arising from a $\ast$-functor $\alpha \colon {\mathcal A}\rightarrow {\mathcal B}$.

\section{Equivalences}

Let $\mathcal A$ and $\mathcal B$ be unital graded $C^\ast$-categories, and let $\varphi, \psi \colon {\mathcal A}\rightarrow {\mathcal B}$ be graded $\ast$-functors. As in \cite{Mitch2.5}, we say $\varphi$ and $\psi$ are {\em unitarily equivalent} if for each object $a\in \Ob ({\mathcal A})$ we have an even unitary element $U_a \in \Hom (\varphi (a),\psi (a))_{\mathcal B}$ such that for each element $x\in \Hom (a,b)_{\mathcal A}$ we have that $U_b \varphi (x) = \psi (x) U_a$.

\begin{definition}
We call a graded $\ast$-functor $\varphi \colon {\mathcal A}\rightarrow {\mathcal B}$ an {\em even unitary equivalence} if there is a $\ast$-functor $\psi \colon {\mathcal B}\rightarrow {\mathcal A}$ such that the compositions $\varphi \circ \psi$ and $\psi \circ \varphi$ are the unitarily equivalent to the identity $\ast$-functor.
\end{definition}

\begin{theorem}
Let $\mathcal A$, $\mathcal B$, and $\mathcal C$ be unital graded $C^\ast$-categories, and let $\varphi , \psi \colon {\mathcal B}\rightarrow {\mathcal C}$ be graded $\ast$-homomorphisms. Suppose that $\varphi$ and $\psi$ are unitarily equivalent. Then the functors $\psi_\ast , \varphi_\ast \colon E({\mathcal A},{\mathcal B})\rightarrow E({\mathcal A},{\mathcal C})$ are equal.
\end{theorem}

\begin{proof}
Let $U \in \Hom (c,d)_{\mathcal C}$ be a unitary element. Then as noted in \cite{Mitch2.5} we have a path of unitaries between the two elements
\[ \left( 
\begin{array}{cc}
0 & U^\ast  \\
U & 0 \\
\end{array} \right) \quad \textrm{and} \quad 1 \]
in the $C^\ast$-algebra $\Hom (c\oplus d,c\oplus d)_{\mathcal C}$.

Let $x\in \Hom (a,b)_{\mathcal B}$. Then
\[ \left( \begin{array}{cc}
0 & U_b^\ast  \\
U_b & 0 \\
\end{array} \right) \varphi (x)\oplus 1 \left( \begin{array}{cc}
0 & U_a^\ast  \\
U_a & 0 \\
\end{array} \right) = 1\oplus \psi (x) \]

Hence the maps $\varphi \oplus 1$ and $1\oplus \psi$ are homotopic on each morphism set via the above paths of unitaries. It follows that the additive $\ast$-functors $\varphi_\oplus , \phi_\oplus \colon {\mathcal B}_\oplus \rightarrow {\mathcal C}_\oplus$ are homotopic. The result now follows.
\end{proof}

The following is immediate.

\begin{corollary}
Let $\mathcal A$, $\mathcal B$, and $\mathcal C$ be unital graded $C^\ast$-categories, and let $\varphi \colon {\mathcal B} \rightarrow {\mathcal C}$ be a unitary equivalence. Then the map $\varphi_\ast \colon E({\mathcal A},{\mathcal B}) \rightarrow E({\mathcal A},{\mathcal C})$ is an isomorphism.
\end{corollary}

We have a similar result in the other variable.

\begin{proposition}
Let $\mathcal A$, $\mathcal B$, and $\mathcal C$ be graded $C^\ast$-categories, and let $\varphi , \psi \colon {\mathcal A}\rightarrow {\mathcal B}$ be graded $\ast$-homomorphisms. Suppose that $\varphi$ and $\psi$ are unitarily equivalent. Then the functors $\psi^\ast , \varphi^\ast \colon E({\mathcal B},{\mathcal C})\rightarrow E({\mathcal A},{\mathcal C})$ are equal.
\end{proposition}

\begin{corollary}
Let $\mathcal A$, $\mathcal B$, and $\mathcal C$ be graded $C^\ast$-categories, and let $\varphi \colon {\mathcal B} \rightarrow {\mathcal C}$ be a unitary equivalence. Then the map $\varphi^\ast \colon E({\mathcal B},{\mathcal C}) \rightarrow E({\mathcal A},{\mathcal C})$ is an isomorphism.
\end{corollary}

\begin{definition}
Let $\mathcal A$ be a $C^\ast$-category. We call $\mathcal A$ {\em additive} if:

\begin{itemize}

\item We have a {\em zero object} $0\in \Ob ({\mathcal A})$ such that for any object $a\in \Ob ({\mathcal A})$ we have unique morphisms $0\in \Hom (0,a)$ and $0\in \Hom (a,0)$.

\item For any two objects $a,b\in \Ob ({\mathcal A})$ we have an object $a\oplus b$ equipped with morphisms $i_a \in \Hom (a,a\oplus b)$, $i_b\in \Hom (b,a\oplus b)$, $p_a\in \Hom (a\oplus b,a)$, and $p_b\in \Hom (a\oplus b,b)$ such that $p_ai_a =1_a$, $p_bi_b=1_b$, and $i_ap_a+i_bp_b =1_{a\oplus b}$.

\end{itemize}

\end{definition}

It is clear from the definition that for any $C^\ast$-category $\mathcal A$, the additive completion ${\mathcal A}_\oplus$ is additive.

\begin{proposition}
Let $\mathcal A$ be a unital additive $C^\ast$-category. Then ${\mathcal A}_\oplus$ is unitarily equivalent to $\mathcal A$.
\end{proposition}

In particular, the category $({\mathcal A}_\oplus )_\oplus$ is unitarily equivalent to $\mathcal A$. Moreover, the $C^\ast$-category $\mathcal K$ is additive, so for any $C^\ast$-category $\mathcal B$ the categories ${\mathcal B}\gtp {\mathcal K}$, ${\mathcal B}_\oplus \gtp {\mathcal K}$, and $({\mathcal B}\gtp {\mathcal K})_\oplus$ are all equivalent. This yields the slightly simpler formulation of the $E$-theory groups
\[ E({\mathcal A},{\mathcal B}) =  \llbracket S \widehat{\otimes} \mathcal{A} \gtp \mathcal{K} , \mathcal{B} \gtp \mathcal{K} \rrbracket .\]

We shall use this formulation without further comment when convenient from now on.

\begin{theorem}[Homotopy invariance]
Suppose $\alpha, \beta \colon \mathcal{A} \rightarrow \mathcal{B}$ are homotopic graded $\ast$-functors and $\mathcal{C}$ is a $C^\ast$-category. Then the induced maps 
\begin{enumerate}
\item $\alpha_{\ast}, \beta_{\ast} \colon E(\mathcal{C},\mathcal{A}) \rightarrow E(\mathcal{C},\mathcal{B})$, and 
\item $\alpha^{\ast}, \beta^{\ast} \colon E(\mathcal{B},\mathcal{C}) \rightarrow E(\mathcal{A},\mathcal{C})$
\end{enumerate}
 are equal.
\end{theorem}

\begin{proof}
\begin{enumerate}
\item Let $f \colon S \widehat{\otimes} \mathcal{C} \widehat{\otimes}\mathcal{K} \dashrightarrow \mathcal{A} \widehat{\otimes} \mathcal{K}$ be a graded asymptotic morphism representing $\llbracket f \rrbracket \in E(\mathcal{C},\mathcal{A})$, then we can define $\alpha_\ast$ and $\beta_\ast$ by 
\[\alpha_\ast (f_t) = (\alpha \widehat{\otimes} \text{id}_{\mathcal{K}} ) \circ f_t ,\]
\[\beta_\ast (f_t) =  (\beta \widehat{\otimes} \text{id}_{\mathcal{K}} ) \circ f_t .\]
Since $\alpha$ and $\beta$ are homotopic, and hence $\alpha_\ast$ and $\beta_\ast$ are equal since we are homotopic on representatives of our class. 
\item Similarly to (1) we can define $\alpha^\ast$ and $\beta^\ast$, for an representative $g$ of $\llbracket g \rrbracket \in E(\mathcal{B},\mathcal{C})$ by
\[\alpha^\ast (g_t) = g_t \circ (\text{id}_{\mathcal{S}} \widehat{\otimes}\alpha  \widehat{\otimes} \text{id}_{\mathcal{K}}) ,\]
\[\beta^\ast(g_t) = g_t \circ (\text{id}_{\mathcal{S}} \widehat{\otimes}\beta \widehat{\otimes} \text{id}_{\mathcal{K}}).\]
Then again since $\alpha$ and $\beta$ are homotopic, we can clearly see that $\alpha^\ast$ and $\beta^\ast$ are equal at the level of homotopy classes. 
\end{enumerate}
\end{proof}

The following immediately follow from the fact that we have an isomorphism $\mathcal{K} \gtp \mathcal{K} \cong \mathcal{K}$. 
\begin{proposition}
Let $\mathcal{A}$ and $\mathcal{B}$ be graded $C^\ast$-categories. Then we have natural isomorphisms 
\[E(\mathcal{A} \otimes \mathcal{K}, \mathcal{B}) \cong E(\mathcal{A}, \mathcal{B}) ~ \text{and} ~ E(\mathcal{A}, \mathcal{B} \otimes \mathcal{K}) \cong E(\mathcal{A}, \mathcal{B})\]
\end{proposition}

Finally, just as in the case of ordinary $E$-theory and $KK$-theory, we have the following abstract form of duality.

\begin{theorem} \label{duality}
Let $\mathcal D$ and $\mathcal E$ be graded $C^\ast$-categories with $E$-theory elements $\alpha \in E({\mathcal D},{\mathcal E})$ and $\beta \in E({\mathcal E},{\mathcal D})$ such that the product of $\alpha$ and $\beta$ is the identity on $E({\mathcal D},{\mathcal D})$ and the product of $\beta$ and $\alpha$ is the identity on $E({\mathcal E},{\mathcal E})$.

Then for graded $C^\ast$-categories $\mathcal A$ and $\mathcal B$ we have natural isomorphisms
\[ E({\mathcal A} \gtp {\mathcal D} , {\mathcal B})\cong E({\mathcal A}, {\mathcal B}\gtp {\mathcal E}) \]
and
\[ E({\mathcal A} \gtp {\mathcal E} , {\mathcal B})\cong E({\mathcal A}, {\mathcal B}\gtp {\mathcal D}) \]
\end{theorem}

\section{Bott Periodicity}


Now in $E$-theory we have a Bott periodicity result that can be found in \cite{GH04} for the complex graded case and in \cite{Bro16} for the real case. Before we state the result we note that $\mathbb{F}_{p,q}$ is the \emph{Clifford algebra} on generators $e_1,e_2, \ldots e_n$ such that $e_i^2=1$, $f_j^2=-1$,  $e_i e_j= -e_j e_i$, $f_if_j =-f_jf_i$, and $e_if_j =-f_je_i$ for all $i,j$. We define a grading by deeming the generators $e_i$ and $f_j$ to be odd. Then we have the following:

\begin{theorem}\label{graded Bott map}
There is a $\ast$-homomorphism $b\colon S \rightarrow \Sigma \widehat{\otimes} \mathbb{F}_{1,0}$ inducing an isomorphism 
which induces an invertible $E$-theory element
\[ [b]\in E(\F, \Sigma \widehat{\otimes} \mathbb{F}_{1,0}) . \]
\end{theorem}

In particular, the $\ast$-homomorphism $b\colon S \rightarrow \Sigma \widehat{\otimes} \mathbb{F}_{1,0}$ induces an invertible $E$-theory element
\[ [b]\in  E(\F , \Sigma \widehat{\otimes} \mathbb{F}_{1,0}) .\]
Let $\mathcal A$ and $\mathcal B$ be $C^\ast$-categories. Then we have a $\ast$-functor
\[ b_{\mathcal B} \colon S{\mathcal B}\rightarrow \Sigma \widehat{\otimes} \mathcal B \widehat{\otimes} \mathbb{F}_{1,0} = \Sigma {\mathcal B} \widehat{\otimes} \mathbb{F}_{1,0} \]
defined by taking the tensor product with the identiy $\ast$-functor on $\mathcal B$. This $\ast$-functor also, by construction of the $E$-theory product, yields an isomorphism
\[ [b_{\mathcal B}] \in E( {\mathcal B}, \Sigma {\mathcal B}\widehat{\otimes} \mathbb{F}_{1,0}) .\]

Taking the $E$-theory product of an element of $E({\mathcal A},{\mathcal B})$ with  $[b_{\mathcal B}]$ gives us a natural isomorphism of $E$-theory groups
\[ E({\mathcal A},{\mathcal B}) \rightarrow E({\mathcal A}, \Sigma {\mathcal B}\widehat{\otimes} \mathbb{F}_{1,0}) . \]

There is similarly an isomorphism
\[ E({\mathcal A},{\mathcal B}) \rightarrow E(\Sigma {\mathcal A}\gtp \F_{1,0} , {\mathcal B}) \]

By theorem \ref{duality} we also have natural isomorphisms
\[ E({\mathcal A}\gtp \F_{1,0} ,{\mathcal B}) \cong E({\mathcal A}, \Sigma {\mathcal B}) \qquad E( \Sigma {\mathcal A}, {\mathcal B}\gtp \F_{0,1}) \]

To use this result we need a few facts on Clifford algebras. Observe
\[ \F_{p,q}\gtp \F_{r,s} \cong \F_{p+r,q+s} \]

Clifford algebras have the following form of periodicity \cite{LM} in the complex and real cases respectively
\[ \C_{2,0}\cong \C_{0,2} \cong \C_{1,1} \qquad \R_{8,0}\cong \R_{0,8}\cong \R_{4,4} \]

To apply these results to $E$-theory, note the following:

\begin{proposition}
We have an isomorphism $\K \gtp \F_{1,1} \equiv \K$.
\end{proposition}

\begin{proof}
Need to say something here.
\end{proof}

It follows that given $C^\ast$-categories $\mathcal A$ and $\mathcal B$ we have natural isomorphisms
\[ E({\mathcal A},{\mathcal B}) \cong E({\mathcal A}\gtp \F_{1,1},{\mathcal B}) \cong E({\mathcal A}, {\mathcal B}\gtp F_{1,1} )\cong E({\mathcal A}\gtp \F_{1,1} , {\mathcal B}\gtp \F_{1,1}) \]

Hence we have natural isomorphisms
\[ E({\mathcal A}\gtp \F_{0,1},{\mathcal B})\cong E(\Sigma {\mathcal A} \gtp \F_{1,1} ,{\mathcal B}) \cong E(\Sigma {\mathcal A},{\mathcal B}) \]
and
\[ E({\mathcal A},{\mathcal B}\gtp \F_{0,1} ) \cong E({\mathcal A},\Sigma {\mathcal B}) \]

By Bott periodicity and periodicity of Clifford algebras, in the complex case it follows that we have a natural isomorphism
\[ E({\mathcal A},{\mathcal B})\cong E({\mathcal A},{\mathcal B}\gtp \C_{2,0} )\cong E({\mathcal A},\Sigma^2 {\mathcal B} ) \]

Similarly, in the real case we have a natural isomorphism
\[ E({\mathcal A},{\mathcal B})\cong E({\mathcal A},{\mathcal B}\gtp \R_{8,0} )\cong E({\mathcal A},\Sigma^8 {\mathcal B} ) \]

\begin{proposition}
Let $\mathcal A$ and $\mathcal B$ be $C^\ast$-categories. We a natural isomorphism
\[ E({\mathcal A}\gtp \F_{p,q},{\mathcal B}) \cong E({\mathcal  A}, {\mathcal B}\gtp \F_{p,q}) \]
\end{proposition}

\begin{proof}
By Bott periodicity we have
\[ E({\mathcal A}\gtp \F_{1,0},{\mathcal B}) \cong E({\mathcal  A}, \Sigma {\mathcal B}) \cong E({\mathcal A},{\mathcal B}\gtp \F_{1,0}) \]

By periodicity of Clifford algebras, whether real or complex, we have
\[ E({\mathcal A}\gtp \F_{0,8} ,{\mathcal B}) \cong E({\mathcal A}\gtp \F_{4,4} ,{\mathcal B}) \cong E({\mathcal A},{\mathcal B}) \]
and
\[ E({\mathcal A},{\mathcal B}\gtp \F_{0,8} ) \cong E({\mathcal A},{\mathcal B}) \]

Hence, by properties of tensor products of Clifford algebras
\[ E({\mathcal A}\gtp \F_{0,1},{\mathcal B})\cong E({\mathcal A} \gtp F_{7,8} ,{\mathcal B}) \cong E({\mathcal A}\gtp \F_{15,0}, {\mathcal B}) .\]

Similarly $E({\mathcal A},{\mathcal B}\gtp \F_{0,1}) \cong E({\mathcal A},{\mathcal B}\gtp \F_{15,0})$. So by the above
\[ E({\mathcal A}\gtp \F_{0,1},{\mathcal B}) \cong E({\mathcal  A}, {\mathcal B}\gtp \F_{0,1}) \]

The result now follows.
\end{proof}

Combining the above proposition with Bott periodicty, we immediately have the following

\begin{corollary} \label{BP2}
We have a natural isomorphism
\[E(\mathcal{A}, \mathcal{B}) \rightarrow E(\Sigma \mathcal{A}, \Sigma \mathcal{B}) \]
\end{corollary}

\section{Long Exact Sequences}

\begin{definition}
Let $\mathcal J$, $\mathcal B$, and $\mathcal A$ be $C^\ast$-categories with the same set of objects. Let $i\colon {\mathcal J}\rightarrow {\mathcal B}$ and $j\colon {\mathcal B}\rightarrow {\mathcal A}$ be $\ast$-functors such that $i(a)=a$ and $j(a)=a$ for each object $a\in \Ob ({\mathcal A})$. Then we call the sequence of $\ast$-functors
\[ 0 \rightarrow {\mathcal J}\stackrel{i}{\rightarrow} {\mathcal B}\stackrel{j}{\rightarrow} {\mathcal A}\rightarrow 0 \]
a {\em short exact sequence} if for all $a,b\in \Ob ({\mathcal J})$ the sequence of morphism sets
\[ 0 \rightarrow \Hom (a,b)_{\mathcal J}\stackrel{i}{\rightarrow} \Hom(a,b)_{\mathcal B}\stackrel{j}{\rightarrow} \Hom (a,b)_{\mathcal A}\rightarrow 0 \]
is a short exact sequence of abelian groups.
\end{definition}
If 
\[ 0 \rightarrow {\mathcal J}\stackrel{i}{\rightarrow} {\mathcal B}\stackrel{j}{\rightarrow} {\mathcal A}\rightarrow 0 \]
is a short exact sequence of $C^\ast$-categories, the $C^\ast$-category $\mathcal J$ is isomorphic under the $\ast$-functor $i$ to a $C^\ast$-ideal in $\mathcal B$, and the $C^\ast$-category $\mathcal A$ is isomorphic to the quotient ${\mathcal B}/{\mathcal J}$, with the map $j$ the quotient map. The proof is similar to the corresponding result for abelian groups.
\begin{proposition}\label{SESgivesasy}
Let 
\[ 0\rightarrow {\mathcal J}\rightarrow {\mathcal B}\stackrel{\pi}{\rightarrow} {\mathcal A}\rightarrow 0 \]
be a short exact sequence of $C^\ast$-categories, where the $C^\ast$-category $\mathcal J$ is a $C^\ast$-ideal in the category $\mathcal B$. Let $\{ u_t^a \ |\ t\in [1,\infty ), a\in \Ob ({\mathcal B}) \}$ be a quasi-central set of approximate units for the pair $({\mathcal B},{\mathcal J})$. Then there is an asymptotic functor $\sigma \colon \Sigma {\mathcal A}\dashrightarrow {\mathcal J}$ such that if $s\colon \Hom (a,b)_{\mathcal A}\rightarrow \Hom (a,b)_{\mathcal B}$ is a set of continuous sections, then
\[ \lim_{t\rightarrow \infty} \| \sigma_t (f\otimes x) - f(u_t)s(x) \| =0 \]
for all $f\in \Sigma$, $x\in \Hom (a,b)_{\mathcal A}$.
\end{proposition}
\begin{proof}
Let $\Sigma = \Sigma {\mathbb F}$. Observe that $\Sigma {\mathcal A} = \Sigma \otimes {\mathcal A}$. Then by Lemma \ref{techexact}, we can define a $\ast$-functor $\sigma \colon \Sigma {\mathcal A}\dashrightarrow {\mathcal J}$ by writing
\[ \sigma_t (f\otimes x) = f(u_t)s(x) . \]
Certainly
\[ \lim_{t\rightarrow \infty} \| \sigma_t (f\otimes x) - f(u_t)s(x) \| =0 .\]
Let $s'\colon \Hom (a,b)_{\mathcal B}\rightarrow \Hom (a,b)_{\mathcal A}$ be a different set of continuous sections. Define $\sigma' \colon \Sigma {\mathcal A}\dashrightarrow {\mathcal J}$ by writing $\sigma'_t (f\otimes x) = f(u_t)s'(x)$. Then again by  lemma \ref{techexact}, 
\[ \lim_{t\rightarrow \infty} \|  f(u_t)(s(x)-s'(x)) \| =0 \]
In particular, it follows that
\[ \lim_{t\rightarrow \infty} \| \sigma_t (f\otimes x) - f(u_t)s'(x) \| =0 \]
and we are done.
\end{proof}

In order to show that $E$-theory is half exact in this framework, we first show that we have exactness when we consider the mapping cylinder $\mathcal{C}_{\alpha}$ of the map $\alpha \colon \mathcal{A} \rightarrow \mathcal{A}/\mathcal{J}$ gives an exact sequence and then use some theory to show that for an ideal we have that $J$ and $\mathcal{C}_{\alpha}$ are inverses in the $E$-theory category.
\begin{lemma}\label{exactinthemiddle}
For a graded $C^\ast$-category $\mathcal{B}$, the following are exact in the middle. 
\begin{enumerate}
\item $E(\mathcal{B},\mathcal{C}_\alpha) \xrightarrow{\beta_\ast} E(\mathcal{B},\mathcal{A}) \xrightarrow{\alpha_\ast} E(\mathcal{B},\mathcal{A}/\mathcal{J})$ is exact,
\item $E(\mathcal{A}/\mathcal{J},\mathcal{B}) \xrightarrow{\alpha^\ast} E(\mathcal{A},\mathcal{B}) \xrightarrow{\beta^\ast} E(\mathcal{C}_\alpha,\mathcal{B})$ is exact.
\end{enumerate}
\end{lemma}
\begin{proof} 
\begin{enumerate}
\item Let $\varphi_t \colon S \widehat{\otimes} \mathcal{B} \widehat{\otimes} \mathcal{K} \dashrightarrow \mathcal{A} \widehat{\otimes} \mathcal{K}$, and $[\varphi_t] \in E_g(\mathcal{B}, \mathcal{A})$ be such that $\alpha_\ast[\varphi_t] =0$, i.e $[\varphi_t] \in \text{Ker}~ {\alpha}_\ast$. Then $\alpha \circ \varphi_t \sim_h 0$. 
Now let $\theta \colon S \widehat{\otimes} \mathcal{B} \widehat{\otimes} \mathcal{K} \dashrightarrow \mathcal{A} \widehat{\otimes} \mathcal{K} \widehat{\otimes} C[0,1]$ be such that $\alpha \circ\theta_t$ is a homotopy between $\alpha \circ \varphi_t$ and $0$. Then write 
\[\theta \colon  S \widehat{\otimes} \mathcal{B} \widehat{\otimes} \mathcal{K} \dashrightarrow \mathcal{A} \widehat{\otimes} C_0[0,1)  \widehat{\otimes} \mathcal{K} \cong C\mathcal{A} \widehat{\otimes} \mathcal{K},\]
and define $\psi \colon S \widehat{\otimes} \mathcal{B} \widehat{\otimes} \mathcal{K} \dashrightarrow \mathcal{C}_{\alpha} \widehat{\otimes} \mathcal{K}$ on objects as the identity and on morphisms such that
\[ \psi_t = \varphi_t \oplus \theta_t.\] 
Then $\beta \circ \psi_t = \varphi_t$, and $\beta_{\ast} [\psi_t] = [\varphi_t]$ and so $[\varphi_t] \in \text{Im} ~\beta_{\ast}$. So $\text{Ker} \alpha_{\ast} \subseteq \text{Im}~ \alpha_{\ast}$.
Conversely, from homological algebra we know that $\text{Im}~ \beta_{\ast} \subseteq \text{Ker} ~\alpha_{\ast}$ is equivalent to $\alpha_{\ast} \beta_{\ast}=0$. 
Now for $x_t \in E(\mathcal{B},\mathcal{C}_{\alpha})$
\begin{align*}
\alpha_{\ast}\beta_{\ast}(x_t) 
& = \alpha_{\ast}(\beta \widehat{\otimes} \text{id}_{\mathcal{K}} \circ x_t) \\
& = \alpha\widehat{\otimes} \text{id}_{\mathcal{K}} \circ\beta\widehat{\otimes} \text{id}_{\mathcal{K}} \circ x_t  \\
& = (\alpha \circ \beta) \widehat{\otimes} \text{id}_{\mathcal{K}} \circ x_t  \\
& = 0
\end{align*}
since $\alpha \circ \beta$ is equal to the identity. 
\item
Firstly $\beta^\ast \circ \alpha^\ast = 0$ so $\text{Im}~ 
\alpha^{\ast} \subseteq \text{Ker} ~\beta^{\ast}$ by a similar method to part(1). Now we show that $\text{Ker} ~\beta^{\ast} \subseteq \text{Im}~ 
\alpha^{\ast}$.
Let $\varphi_t \colon S \widehat{\otimes} \mathcal{A} \widehat{\otimes} \mathcal{K} \dashrightarrow \mathcal{B} \widehat{\otimes} \mathcal{K}$ and $q \colon \mathcal{C}_
\alpha \widehat{\otimes} \mathcal{K} \rightarrow \mathcal{A} \widehat{\otimes} \mathcal{K}$ be the identity on objects and a projection on morphisms. Then we want $\theta_t \colon S \widehat{\otimes} \Sigma \mathcal{A}/\mathcal{J} \widehat{\otimes} \mathcal{K} \dashrightarrow \Sigma \mathcal{B} \widehat{\otimes} \mathcal{K}$ such that 
\[[\theta_t \circ (\text{id}_{\mathcal{S}} \circ \Sigma 
\alpha)] = \Sigma[\varphi] \in E(\Sigma \mathcal{A}, \Sigma \mathcal{B}),\]
where $\alpha \colon S \widehat{\otimes} \mathcal{A} \widehat{\otimes} \mathcal{K} \rightarrow  S \widehat{\otimes} \mathcal{A}/\mathcal{J} \widehat{\otimes} \mathcal{K}$.
Now let $\eta \colon S \widehat{\otimes} \mathcal{C}_\alpha \widehat{\otimes} \mathcal{K} \rightarrow I\mathcal{B} \widehat{\otimes} \mathcal{K}$ (where $I=[0,1]$) be a homotopy between $\varphi_t \circ (\id \widehat{\otimes} q)$ and $0$. Then by symmetry of homotopy we can obtain
\[\widetilde{\eta_{t}} \colon S \widehat{\otimes} \mathcal{C}_\alpha \widehat{\otimes} \mathcal{K} \rightarrow I_1\mathcal{B} \widehat{\otimes} \mathcal{K},\]
where $I_1=[0,1]$.
Now we also have an inclusion 
\[i \colon \Sigma S \widehat{\otimes} \mathcal{A}/\mathcal{J} \widehat{\otimes} \mathcal{K} \cong S \widehat{\otimes} \Sigma \mathcal{A}/\mathcal{J} \widehat{\otimes} \mathcal{K} \rightarrow S \widehat{\otimes} \mathcal{C}_{\alpha}\widehat{\otimes} \mathcal{K},\]
defined on objects as the identity and on morphisms by
\[i(g \widehat{\otimes} f \widehat{\otimes} k) = g \widehat{\otimes} (0 \oplus f) \widehat{\otimes} k,\]
where $g \in S$, $f \colon [-1,1] \rightarrow \mathcal{A}/\mathcal{J}$  and $k \in \mathcal{K}$. 
Then $(\id_S\widehat{\otimes} q) \circ i = 0$.
Define $\theta_t \colon S \widehat{\otimes} \Sigma  \mathcal{A}/\mathcal{J} \widehat{\otimes} \mathcal{K} \dashrightarrow \Sigma \mathcal{B} \widehat{\otimes} \mathcal{K}$ by $\theta_t = \widetilde{\eta_t} \circ i$. Then we need to show that $\theta_t \circ (\id_S \circ \Sigma \alpha)$ is homotopic to $\Sigma \varphi$. 
Now $\Sigma \varphi_t = \widetilde{\eta_t} \circ i \circ (\id_S \circ \Sigma \alpha)$, and so 
\[ ||\theta_t \circ (\text{id}_S\circ \Sigma \alpha) - (\Sigma \varphi_t) || \to 0,\] 
as $t \to \infty$ and asymptotic equivalence implies homotopic equivalence by Proposition~\ref{Equivalent implies homotopic}, half exactness follows. 
\end{enumerate}
\end{proof}

\begin{proposition}\label{twoSESgivesquarediagram}
Let 
\[\xymatrixcolsep{3pc}\xymatrixrowsep{3pc}\xymatrix{
0 \ar[r] & \mathcal{J}_0 \ar[r] \ar[d] & \mathcal{B}_0 \ar[r] \ar[d] & \mathcal{A}_0\ar[d]  \ar[r] &0\\
0 \ar[r] & \mathcal{J}_1\ar[r]^{} &\mathcal{B}_1 \ar[r]^{} &\mathcal{A}_1 \ar[r] & 0}\]
be a commuting diagram of short exact sequences of separable $C^\ast$-categories. 
Then this gives a commuting diagram  
\[\xymatrixcolsep{3pc}\xymatrixrowsep{3pc}\xymatrix{
 \Sigma \mathcal{A}_0  \ar[r] \ar[d] & \mathcal{J}_0 \ar[d] \\
\Sigma \mathcal{A}_1  \ar[r] & \mathcal{J}_1 }\]
in the $E$-theory category.
\end{proposition}
The proof follows precisely from that of Proposition 5.8 in~\cite{GHT}, since all $C^\ast$-categories have the same objects, the morphisms sets have approximate units indexed the same and the morphisms satisfy the properties that the short exact sequence has for $C^\ast$-algebras.  
Likewise the Proposition below is proves just as in the case of Proposition~5.9 in~\cite{GHT}.

\begin{proposition}
Let $\sigma \colon \Sigma \mathcal{A} \dashrightarrow \mathcal{J}$ be the asymptotic functor associated to the short exact sequence
\[ 0 \rightarrow \mathcal{J} \rightarrow \mathcal{B} \rightarrow \mathcal{A} \rightarrow 0.\]
Then if $\mathcal{D}$ is a separable $C^\ast$-category and if 
\[ \sigma_{\mathcal{D}} \colon \Sigma \mathcal{A} \widehat{\otimes} \mathcal{D} \dashrightarrow \mathcal{J} \widehat{\otimes} \mathcal{D}\] 
is the morphism associated to the short exact sequence
\[ 0 \rightarrow \mathcal{J} \widehat{\otimes} \mathcal{D}\rightarrow \mathcal{B}\widehat{\otimes} \mathcal{D} \rightarrow \mathcal{A} \widehat{\otimes} \mathcal{D}\rightarrow 0,\]
then $\sigma_{\mathcal{D}}$ is equal to $\sigma \otimes \id_\mathcal{D}$.
\end{proposition}
The following is just taken from~\cite{GHT}; for a proof see Proposition 5.11. For $\mathbb{C}$ as a $C^\ast$-category, so viewed as a one object $C^\ast$-category, we obtain a natural asymptotic functor.
\begin{proposition}
The asymptotic morphism $\sigma \colon \Sigma \mathbb{C} \dashrightarrow C \mathbb{C}$ coming from the short exact sequence 
\[ 0 \rightarrow \Sigma \mathbb{C} \rightarrow C \mathbb{C} \rightarrow \mathbb{C} \rightarrow 0,\]
is just the identity morphism in the $E$-theory category. 
\end{proposition}

Let $\alpha \colon \mathcal{B} \rightarrow \mathcal{A}$ be a $\ast$-functor, and $\mathcal{C}_{\alpha}$ be the associated mapping cylinder. Let $\mathcal{J} \subset \mathcal{B}$ be a $C^\ast$-ideal with quasi-central approximate unit. Then consider the inclusion $\ast$-functor, $\tau \colon \mathcal{J} \rightarrow \mathcal{C}_{\alpha}$. Consider the short exact sequence: 
\[ 0 \rightarrow \Sigma \mathcal{J} \rightarrow C \mathcal{B} \rightarrow \mathcal{C}_{\alpha} \rightarrow 0.\]
By proposition~\ref{SESgivesasy}, we have an associated asymptotic functor $\sigma \colon \Sigma \mathcal{C}_{\alpha} \dashrightarrow \Sigma \mathcal{J}$. Then the following result follows from proposition~\ref{twoSESgivesquarediagram}; for exact details see Proposition~5.14 in \cite{GHT}.
\begin{proposition}\label{Jequivtomappingcylinder}
The $\ast$-functor $\Sigma \tau \colon \Sigma \mathcal{J} \rightarrow \Sigma \mathcal{C}_\alpha$ and the asymptotic functor $\Sigma \sigma \colon \Sigma \mathcal{C}_\alpha \dashrightarrow \Sigma \mathcal{J}$ are mutually inverse morphisms in the $E$-theory category.  
\end{proposition}

\begin{lemma}\label{halfexact}
Let $0 \rightarrow \mathcal{J} \xrightarrow{\beta} \mathcal{B} \xrightarrow{\alpha} \mathcal{A} \rightarrow 0$ be a short exact sequence of graded $C^\ast$-categories and let $\mathcal{D}$ be a graded $C^\ast$-category. Then 
\begin{enumerate}
\item $E(\mathcal{D},\mathcal{J}) \xrightarrow{\beta_\ast} E(\mathcal{D},\mathcal{B}) \xrightarrow{\alpha_\ast} E(\mathcal{D},\mathcal{A})$ is exact,
\item $E(\mathcal{A},\mathcal{D}) \xrightarrow{\alpha^\ast} E(\mathcal{B},\mathcal{D}) \xrightarrow{\beta^\ast} E(\mathcal{J},\mathcal{D})$ is exact.
\end{enumerate}
\end{lemma}
\begin{proof}
\begin{enumerate}
\item First we consider the connection between $\mathcal{J}$ and $\mathcal{C}_{\alpha}$. From Lemma~\ref{exactinthemiddle}, with $\alpha \colon \mathcal{B} \rightarrow \mathcal{A}$, it follows that the sequence
\[E(\mathcal{D},\mathcal{C}_{\alpha}) \xrightarrow{\beta_\ast} E(\mathcal{D},\mathcal{B}) \xrightarrow{\alpha_\ast} E(\mathcal{D},\mathcal{A})\] is exact.
Then by Proposition~\ref{Jequivtomappingcylinder}, we have that $E(\Sigma \mathcal{C}_{\alpha}, \Sigma\mathcal{J}) \cong E(\Sigma \mathcal{J}, \Sigma \mathcal{C}_{\alpha})$ and hence by Corollary~\ref{BP2}, $\mathcal{J}$ and $\mathcal{C}_{\alpha}$ are isomorphic in $E$-theory. 
Then half exactness of the required sequence follows from the commutative diagram:
\[\xymatrixcolsep{3pc}\xymatrixrowsep{3pc}\xymatrix{
0 \ar[r] & E(\mathcal{D},\mathcal{C}_{\alpha}) \ar[r] \ar[d]_{\cong} &E(\mathcal{D},\mathcal{B}) \ar[r] \ar@{=}[d] &E(\mathcal{D},\mathcal{A})\ar@{=}[d]  \ar[r] &0\\
0 \ar[r] & E(\mathcal{D},\mathcal{J}) \ar[r]^{\beta_{\ast}} &E(\mathcal{D},\mathcal{B}) \ar[r]^{\alpha_{\ast}} &E(\mathcal{D},\mathcal{J}) \ar[r] & 0}\]
\item Similarly, the proof of 1 works in the contravariant case.
\end{enumerate}
\end{proof}

Now to extend to long exact sequences, we require the following Proposition.
\begin{proposition}\label{contractibleimpliesiso}
Let \[0 \rightarrow \mathcal{J} \xrightarrow{i} \mathcal{B} \xrightarrow{\alpha} \mathcal{A} \rightarrow 0,\] be a short exact sequence of graded $C^\ast$-categories. Then if $\mathcal A$ is contractible, then $i$ induces an isomorphism between $E(\mathcal{D}, \mathcal{J})$ and $E(\mathcal{D}, \mathcal{B})$.
\begin{proof}
Since $\mathcal A$ is contractible, it follows that for any separable $C^\ast$-category $E(\mathcal{D}, \mathcal{A}) = 0$, since $E$ is a homotopy invariant functor, we have that this is equivalent to having $\mathcal{A}$ as the zero category. 
Then applying half exactness to our short exact sequence, and using the above, we have 
\[E(\mathcal{D}, \mathcal{J})  \rightarrow E(\mathcal{D}, \mathcal{B})  \rightarrow 0\]
and so $i_\ast$ is surjective. 
Now we check that $i_\ast$ is injective.
Consider the short exact sequence,
\[0 \rightarrow \Sigma \mathcal{A} \rightarrow \mathcal{C}_{\alpha} \rightarrow \mathcal{B} \rightarrow 0\]
and once again apply half exactness:
\[E(\mathcal{D}, \Sigma \mathcal{A})  \rightarrow E(\mathcal{D}, \mathcal{C}_{\alpha})  \rightarrow E(\mathcal{D}, \mathcal{B}).\]
Then since $\mathcal{A}$ is contractible, $\Sigma \mathcal{A}$ is also contracible and as $J \cong \mathcal{C}_{\alpha}$ in the $E$-theory category, we obtain the exact sequence
\[0 \rightarrow E(\mathcal{D}, \mathcal{J})  \rightarrow E(\mathcal{D}, \mathcal{B}),\]
and $i_\ast$ is injective, and hence $E(\mathcal{D}, \mathcal{J}) \cong E(\mathcal{D}, \mathcal{B})$ as required. 
\end{proof}
\end{proposition}

\begin{theorem}\label{LES}
For the functor $E$ from the category of graded $C^\ast$-categories and $\ast$-functors to the category of abelian groups and group homomorphisms to every short exact sequence of separable graded $C^\ast$-categories 
\[0 \rightarrow \mathcal{J} \rightarrow \mathcal{B} \rightarrow \mathcal{A} \rightarrow 0,\]
we obtain a long exact sequence of abelian groups
\[\cdots \rightarrow E(\mathcal{D}, \Sigma \mathcal{B}) \rightarrow E(\mathcal{D}, \Sigma \mathcal{A}) \xrightarrow{\partial_{\ast}} E(\mathcal{D}, \mathcal{J}) \rightarrow E(\mathcal{D}, \mathcal{B}) \rightarrow E(\mathcal{D}, \mathcal{A}),\]
where the connecting map $\partial_{\ast}$ fits in to the commutative diagram 
\[\xymatrix{E(\mathcal{D}, \Sigma \mathcal{A}) \ar[dr]^{\beta_{\ast}} \ar[rr]^{\delta_{\ast}} & &  E(\mathcal{D}, \mathcal{J}) \\ & E(\mathcal{D}, \mathcal{C}_{\alpha}) \ar[ur]^{\tau_{\ast}} &}\]
\end{theorem} 
\begin{proof}
Since $E$ is half exact by Lemma~\ref{halfexact}, it suffices to show that we have exactness at $E(\mathcal{D}, \Sigma \mathcal{A})$ and $E(\mathcal{D}, \mathcal{J})$. For exactness at $E(\mathcal{D}, \mathcal{J})$ consider the short exact sequence
\[0 \rightarrow \Sigma \mathcal{A} \rightarrow \mathcal{C}_{\alpha} \rightarrow \mathcal{B} \rightarrow 0.\]
Then by half exactness, we have that 
\[E(\mathcal{D}, \Sigma \mathcal{A}) \rightarrow E(\mathcal{D}, \mathcal{C}_{\alpha}) \rightarrow E(\mathcal{D}, \mathcal{B})\]
is exact, and then by equivalence of $\mathcal{J}$ and $\mathcal{C}_{\alpha}$ in the $E$-theory category, it follws that 
\[E(\mathcal{D}, \Sigma \mathcal{A}) \xrightarrow{\partial_{\ast}} E(\mathcal{D}, \mathcal{J}) \rightarrow E(\mathcal{D}, \mathcal{B})\] 
is exact as required. 
Now for exactness at $E(\mathcal{D}, \Sigma \mathcal{A})$, consider the $C^\ast$-category $\mathcal{T} =C_0((-1,0), \mathcal{B})$ with 
\[\Ob(\mathcal{T}) = \{\Ob(\mathcal{A})\}\]
\[ \Hom ((a,b),(a',b')) = \left\{ (x,f) \; \middle|\ \; \Large\begin{subarray}{c} x \; \in \; \Hom (a,a')_{\mathcal A}, \ f(-1)=f(1)= \alpha (x), \ \vspace{2mm}  \\ f\colon  [-1,1] \; \longrightarrow  \; \Hom (b,b')_{\mathcal B} \textrm{ continuous}    
   \end{subarray}\right\} . \]
Then we have a canonical embedding $\kappa \colon \Sigma \mathcal{A} \rightarrow \mathcal{T}$, which induces a map from $\kappa_{\ast} \colon E(\mathcal{D}, \Sigma \mathcal{A}) \rightarrow E(\mathcal{D}, \mathcal{T})$, which is an isomorphism since the quotient of $\mathcal{T}$ by $\Sigma \mathcal{A}$ is contractible and by Proposition~\ref{contractibleimpliesiso}, $\kappa_{\ast}$ is an isomorphism. 
Now $\Sigma \mathcal{B} \cong C_0((-1,0), \mathcal{B})$ and we have a canonical embedding $C_0((-1,0), \mathcal{B}) \rightarrow \mathcal{T}$, which is homotopic to the composition $\Sigma \mathcal{B} \xrightarrow{\Sigma \alpha} \Sigma \mathcal{A} \xrightarrow{\kappa} \mathcal{T}$. Then applying half exactness to the following short exact sequence
\[ 0 \rightarrow C_0((-1,0), \mathcal{B}) \rightarrow \mathcal{T} \rightarrow \mathcal{C}_{\alpha} \rightarrow 0 \]
we obtain the top exact sequence in the diagram 
\[\xymatrixcolsep{3pc}\xymatrixrowsep{3pc}\xymatrix{E(\mathcal{D}, C_0((-1,0), \mathcal{B}))\ar[r] \ar[d]^{\cong} & E(\mathcal{D}, \mathcal{T})\ar[r]\ar[d]& E(\mathcal{D}, \mathcal{C}_{\alpha}) \ar[d]^{\cong} \\
E(\mathcal{D}, \Sigma \mathcal{B})\ar[r]& E(\mathcal{D}, \Sigma \mathcal{A}) \ar[r]& E(\mathcal{D}, \mathcal{J}) }\] 
and hence exactness at $E(\mathcal{D}, \Sigma \mathcal{A})$ follows.
\end{proof}
The contravariant case is proved similarly.

By Bott periodicity and Theorem~\ref{LES}, we obtain the following result. 
\begin{theorem}
Let $\mathcal{D}$ be a separable $C^\ast$-category and let \[0 \rightarrow \mathcal{J} \rightarrow \mathcal{B} \rightarrow \mathcal{A} \rightarrow 0,\] be a short exact sequence of separable $C^\ast$-categories. Then there are six term exact sequences
\[\xymatrixcolsep{3pc}\xymatrixrowsep{3pc}\xymatrix{
E(\mathcal D, \Sigma\mathcal J) \ar[r]& E(  \mathcal D, \Sigma \mathcal B)\ar[r] &E(\mathcal D, \Sigma \mathcal A)\ar[d]\\
E(\mathcal D,\mathcal A) \ar[u] &\ar[l]^{} E(\mathcal D, \mathcal B)  &\ar[l]^{}E( \mathcal D, \mathcal J)} ,\]
and 
\[\xymatrixcolsep{3pc}\xymatrixrowsep{3pc}\xymatrix{
E(\Sigma \mathcal J, \mathcal D) \ar[d] & ( \Sigma\mathcal B, \mathcal D)\ar[l] &( \Sigma \mathcal A, \mathcal D) \ar[l]\\
E(\mathcal A,\mathcal D) \ar[r]^{} &E(\mathcal B,\mathcal D) \ar[r]^{} &E( \mathcal J, \mathcal D)\ar[u]} ,\]
where the boundary maps are defined by Bott periodicity and product with the $E$-theory class of an element $\sigma \in E(\Sigma \mathcal A, \mathcal J)$ associated to the short exact sequence. 
\end{theorem}

%
\bibliographystyle{alpha}
\bibliography{data}
\end{document}